\documentclass[10pt]{article}

\usepackage{geometry}
\usepackage{latexsym}
\usepackage{mathrsfs}
\usepackage{amsfonts}
\usepackage{amssymb}
\usepackage{subfigure}    
\usepackage{graphicx}
\usepackage{epstopdf}
\usepackage{float}
\usepackage{multirow}
\usepackage{bm}
\usepackage{amsmath}
\usepackage{amsthm}
\usepackage{color}
\usepackage[subnum]{cases}
\usepackage{rotating}
\usepackage{enumitem}
\usepackage{accents}
\usepackage{algorithm}
\usepackage{algorithmic}
\usepackage[title]{appendix}
\usepackage{mathtools}
\usepackage{caption}
\usepackage{varwidth}
\usepackage{array}
\usepackage{multicol}

\usepackage{pgf,tikz,pgfplots}
\pgfplotsset{compat=1.15}
\usepackage{mathrsfs}
\usetikzlibrary{arrows}

\usepackage{url}

\makeatletter
\@addtoreset{equation}{section}
\makeatother

\makeatletter
\def \wideubar{\underaccent{{\cc@style\underline{\mskip15mu}}}}
\def \widebar{\accentset{{\cc@style\underline{\mskip10mu}}}}
\makeatother

\graphicspath{{images/}{images/tomo/}}

\definecolor{blue}{rgb}{0,0,0}
\definecolor{red}{rgb}{0.9,0,0}
\definecolor{green}{rgb}{0,0.9,0}
\definecolor{brown}{rgb}{0.6,0.1,0.1}
\definecolor{lightgreen}{rgb}{0.1,0.5,0.1}
\newcommand{\blue}[1]{\begin{color}{blue}#1\end{color}}

\usepackage[colorlinks=true,breaklinks=true,bookmarks=true,urlcolor=blue,
     citecolor=lightgreen,linkcolor=lightgreen,bookmarksopen=false,draft=false]{hyperref}

\newcommand{\cJ}{{\cal J}}
\newcommand{\cF}{{\cal F}}

\newcommand{\cA}{{\cal A}}

\newcommand{\inprod}[2]{\langle #1 , #2 \rangle }

\newcommand{\bc}{\begin{center}}
	\newcommand{\ec}{\end{center}}

\newcommand{\R}{\mathbb R}

\newcommand{\be}{\begin{equation}}
\newcommand{\ee}{\end{equation}}
\newcommand{\beaa}{\begin{eqnarray*}}
	\newcommand{\eeaa}{\end{eqnarray*}}
\newcommand{\ben}{\begin{enumerate}}
	\newcommand{\een}{\end{enumerate}}
\newcommand{\db}{\hspace*{\fill}{\zapf o}}

\newcommand{\bpn}{\begin{proposition}\twlsf}
	\newcommand{\epn}{\db\end{proposition}}
\newcommand{\bdm}{\begin{displaymath}}
\newcommand{\edm}{\end{displaymath}}
\newcommand{\ba}{\begin{array}}
	\newcommand{\ea}{\end{array}}

\def\texitem#1{\par\smallskip\noindent\hangindent 25pt
	\hbox to 25pt {\hss #1 ~}\ignorespaces}
\newcommand{\norm}[1]{\left\lVert#1\right\rVert}

\newtheorem{assumption}{Assumption}

\newtheorem{lemma}{Lemma}
\newtheorem{proposition}{Proposition}

\newtheorem{remark}{Remark}
\newtheorem{theorem}{Theorem}
\newtheorem{property}{Property}
\newtheorem{example}{Example}

\def\inprod#1#2{\langle#1,\,#2\rangle}
\def\Inprod#1#2{\Big\langle#1,\,#2\Big\rangle}
\def\norm#1{\|#1\|}

\begin{document}

\title{An efficient implementable inexact entropic proximal point algorithm for a class of linear programming problems}
\author{
Hong T. M. Chu\thanks{Department of Mathematics, National University of Singapore ({\tt hongtmchu@u.nus.edu}).},\;
Ling Liang\thanks{Department of Mathematics, National University of Singapore ({\tt liang.ling@u.nus.edu}).},\;
Kim-Chuan Toh\thanks{Department of Mathematics, and Institute of Operations Research and Analytics, National University of Singapore ({\tt mattohkc@nus.edu.sg}). This research is supported in part by the  Ministry of Education of Singapore under Academic Research Fund Grant number: R146-000-xxx-xxx.
},\;
and
Lei Yang\thanks{Department of Applied Mathematics, The Hong Kong Polytechnic University, Hung Hom, Kowloon, Hong Kong, China ({\tt yanglei.math@gmail.com}).}
}
\maketitle

\begin{abstract}
We introduce a class of specially structured linear programming (LP) problems, which has favorable modeling capability for important application problems in different areas such as optimal transport, discrete tomography and economics. To solve these generally large-scale LP problems efficiently, we design an implementable inexact entropic proximal point algorithm (iEPPA) combined with an easy-to-implement dual block coordinate descent method as a subsolver. Unlike existing entropy-type proximal point algorithms, our iEPPA employs a more practically checkable stopping condition for solving the associated subproblems while achieving provable convergence. Moreover, when solving the \underline{c}apacity constrained \underline{m}ulti-marginal \underline{o}ptimal \underline{t}ransport (CMOT) problem (a special case of our LP problem), our iEPPA is able to bypass the underlying numerical instability issues that often appear in the popular entropic regularization approach, since our algorithm does not require the proximal parameter to be very small in order to obtain an accurate approximate solution. \blue{Numerous numerical experiments show that our iEPPA is efficient and robust for solving large-scale CMOT problems.}
The experiments on the discrete tomography problem also highlight the potential modeling power of our model.

\vspace{5mm}


\end{abstract}

\section{Introduction}

In this paper,
we introduce a class of specially structured linear programming (LP) problems of the following form:
\begin{equation}\label{cmotproblem}
\begin{aligned}
\min & \;\;\inprod{C}{X} \\[5pt]
\mbox{s.t.} & \;\; X\in \Omega :  = \Big\{ X\in \R^{n_1\times n_2\times n_3} ~:~ \cA^{(i)} (X)  \;=\;
\bm{b}^{(i)}, \quad
i=1,\dots,N, \quad  0\leq X \leq U   \Big\},
\end{aligned}
\end{equation}
where $\langle\cdot,\cdot\rangle$ denotes the standard inner product in $\R^{n_1\times n_2\times n_3}$, $\cA^{(i)} : \R^{n_1\times n_2\times n_3} \to \R^{m_i}$ is a given linear mapping defined by
\begin{equation*}
\cA^{(i)}(X):=\left[ \begin{array}{c} \inprod{A^{(i)}_1}{X}\\ \vdots \\ \inprod{A^{(i)}_{m_i}}{X}
\end{array}\right], \quad
A_j^{(i)}\in \R^{n_1\times n_2\times n_3}, \quad 1\leq j\leq m_i, \quad 1\leq i\leq N,
\end{equation*}
$\bm{b}^{(i)}=(b^{(i)}_1,\dots,b^{(i)}_{m_i})^{\top}\in\mathbb{R}^{m_i}$ ($i=1,\dots,N$), $C\in\R^{n_1\times n_2 \times n_3} $ and $U\in\R_{+}^{n_1\times n_2\times n_3}\cup \{\infty\}^{n_1\times n_2\times n_3}$ are given data. Moreover, the linear mappings $\cA^{(i)}$ ($i=1,\dots,N$) satisfy Assumption \ref{assumption-nonoverlap} below. As we shall see shortly, problem \eqref{cmotproblem} is a generalization of the classical discrete optimal transport problem which has the form: $\min\big\{\inprod{C}{X} : X\in \R^{n_1\times n_2}, \,\sum_{s=1}^{n_2} X_{rs} = a_r, \,r=1,\ldots,n_1, \,\sum_{r=1}^{n_1} X_{rs} =
b_s, \,s=1,\ldots,n_2, \,X\geq0\big\}$, where $\bm{a}:=(a_1, \dots, a_{n_1})^{\top}$ and $\bm{b}:=(b_1, \dots, b_{n_2})^{\top}$ are two given marginals\footnote{\blue{In the paper, the term `marginal' refers to a vector obtained by the sum of entries of a matrix/tensor over an index set.}} in \(n_1\) and \(n_2\)-dimensional simplices, and $C\in \R^{n_1\times n_2}$ is a given cost matrix.

\begin{assumption}\label{assumption-nonoverlap}
For each fixed $i$ ($1\leq i\leq N$), $A_j^{(i)}$  only has binary entries (0 or 1) for $j=1,\dots,m_i$, and the given constraint tensors $\big \{ A_j^{(i)}\;|\; j = 1,\dots,m_i \big\}$ satisfy the property that
\begin{equation*}
A_j^{(i)}\circ A_k^{(i)} = 0, \quad \text{if } ~j\neq k,\; j,\,k = 1,\dots,m_i,
\end{equation*}
where ``$\circ$'' denotes the Hadamard product. 
\end{assumption}

The property stated in Assumption \ref{assumption-nonoverlap} is equivalent to saying that the non-zero patterns of any two distinct constraint tensors $A^{(i)}_j$ and $A^{(i)}_k$ do not overlap in the $i$-block of the linear constraints $\cA^{(i)} (X) = \bm{b}^{(i)}$. Such structures may look unusual at the first glance, but do appear in a few important application problems, for example, the \underline{c}apacity constrained \underline{m}ulti-marginal \underline{o}ptimal \underline{t}ransport (CMOT) problem \blue{with three marginals}, the discrete tomography problem \cite{abraham2017tomographic,bergounioux2018variational,w2006tomography}, the disaggregation of industry-by-industry input-output tables in economics \cite{Holy-Safr}, and reconstructions of unknown inter-bank liabilities with fixed constraints \cite{Grandy-Veraart}; see details on first two examples in the next two paragraphs.
Moreover, as we shall see later, such special structures allow us to design a highly efficient algorithm to solve the corresponding LPs since they can greatly facilitate the computations of the subproblems involved in our algorithm; see Section \ref{secBCD} and Appendix \ref{apd-BCD} for more details.

The discrete \blue{3-marginal} CMOT problem is modeled as follows:
\begin{equation}\label{cmotproblem3d}
\begin{aligned}
&\min\limits_{X\in\mathbb{R}^{n_1 \times n_2 \times n_3}}\,\,\,
 \langle C, \,X\rangle \\
&\hspace{0.8cm} \mathrm{s.t.} \hspace{0.4cm}
\begin{array}{lll}
&{\textstyle\sum_{s,t}}X_{rst} = a_r, ~r = 1, \dots, n_1, \quad
&{\textstyle\sum_{r,t}}X_{rst} = b_s, ~s = 1, \dots, n_2, \\[5pt]
&{\textstyle\sum_{r,s}}X_{rst} = c_t, ~t = 1, \dots, n_3, \quad
&0 \leq X \leq U,
\end{array}
\end{aligned}
\end{equation}
where $\bm{a}:=(a_1, \dots, a_{n_1})^{\top}\in\Sigma_{n_1}$, $\bm{b}:=(b_1, \dots, b_{n_2})^{\top}\in\Sigma_{n_2}$, $\bm{c}:=(c_1, \dots, c_{n_3})^{\top}\in\Sigma_{n_3}$ are three given marginals with $\Sigma_{n_i}$ denoting the $n_i$-dimensional unit simplex for $i=1, \,2, \,3$. When $n_3=1$, the above problem readily reduces to the classical optimal transport problem mentioned in the first paragraph, but with an additional upper bound constraint. It is clear that problem \eqref{cmotproblem3d} falls into the form of \eqref{cmotproblem} with
\begin{equation}\label{eq-specialAj}
\begin{aligned}
A^{(1)}_{j} &= \bm{e}_j^{(1)} \otimes \bm{1}_{n_2} \otimes \bm{1}_{n_3}, \quad j=1,\dots,n_1, \\
A^{(2)}_{j} &= \bm{1}_{n_1} \otimes \bm{e}_j^{(2)} \otimes \bm{1}_{n_3}, \quad j=1,\dots,n_2, \\
A^{(3)}_{j} &= \bm{1}_{n_1} \otimes \bm{1}_{n_2} \otimes \bm{e}_j^{(3)}, \quad j=1,\dots,n_3,
\end{aligned}
\end{equation}
where $\bm{e}_j^{(i)}$ denotes the $j$th unit vector in $\R^{n_i}$ ($i=1, \,2, \,3$), $\bm{1}_{n_i}$ denotes the $n_i$-dimensional vector of all ones for $i=1, \,2, \,3$, and ``$\otimes $'' denotes the tensor product (see the definition at the end of this section). Problem \eqref{cmotproblem3d} was first proposed and studied by Korman and McCann \cite{km2013insights,km2015optimal} in the 2-marginal continuous case\footnote{In the paper, the 2-marginal case means that we consider problem \eqref{cmotproblem3d} in the matrix case (namely, $n_3=1$).} as an important variant of the classical 2-marginal optimal transport (OT) problem. This variant takes into account limits on the transport capacities\footnote{This consideration can date back to \cite{l1984pro}, and possibly earlier.} via imposing a proper upper bound constraint $X \leq U$, and hence \blue{it is better able to} model some real-life situations. Moreover, we note that if the constraints in \eqref{cmotproblem3d} are summed over a single index instead of two indices (for example, $\sum_s X_{rst} = a_{rt}$ for $r=1,\ldots,n_1$, $t=1,\ldots,n_3$), then the resulting problem can model a multi-commodity flow problem on a bipartite graph, where the commodities are indexed by $t=1,\ldots,n_3$; see, for example, \cite{ks1977an}.

In the 2-dimensional discrete tomography problem studied in \cite{w2006tomography}, one is given the marginals obtained from an $n\times n$ matrix (for simplicity, we discuss the matrix case instead of a third-order tensor) by summing its entries along different directions, for example, $0^\circ$, $45^\circ$, $90^\circ$ and $135^\circ$ directions.
In this case, the formulation in \cite{abraham2017tomographic} would require a sixth-order tensor to model the problem as a 6-marginal optimal transport problem. Unfortunately, this approach leads to an exponential increase in the computational cost because of the curse of dimensionality brought about by the extra dimensions introduced in the higher-order tensor. But using our model in \eqref{cmotproblem}, the variable remains as a matrix and the projections along the four directions are formulated as four blocks of linear constraints, each represented by a linear mapping $\cA^{(i)}(X) = \bm{b}^{(i)}$ with $\bm{b}^{(i)}$ being the given $i$th marginal for $i=1,\dots\!,4$. Moreover, it is not hard to verify that the constraint matrices associated with each linear mapping $\cA^{(i)}$ satisfy Assumption \ref{assumption-nonoverlap}. For the construction of a block of linear constraints that represents a projection along a specific direction, we refer the reader to subsection \ref{numsec-tomo} and Appendix \ref{appen3}.

Note that problem \eqref{cmotproblem} has $n_1n_2n_3$ box-constrained variables and $\sum_{i=1}^N m_i$ linear equality constraints, and thus it is usually a very large-scale LP problem when the dimension of the variable \textit{or} the number of blocks of linear constraints is large. Therefore, classical LP methods such as the simplex method and the interior point method may no longer be efficient enough \textit{or} may consume too much memory when solving this problem. Recently, an entropic regularized approach was proposed in \cite{bccnp2015iterative} to {\bf approximately} solve problem \eqref{cmotproblem3d} in the 2-marginal case with impressive numerical performance. This approach basically modifies the original LP problem by adding an entropic regularization to the objective, and then applies a certain efficient first-order method to solve the resulting computationally more tractable regularized problem to obtain an approximate solution of the original LP problem. In \cite{bccnp2015iterative}, \underline{Dy}kstra's algorithm\footnote{More details on Dykstra's algorithm and its Bregman extension can be found in \cite{bl2000dykstras,d1983algorithm}.} with \underline{K}ullback-\underline{L}eibler projections (DyKL) is adapted to solve the entropic regularized counterpart of problem \eqref{cmotproblem3d} in the 2-marginal case (see \eqref{encapotpro}). This algorithm can be highly efficient if a crude approximate solution is adequate, in which case the regularization parameter needs not be very small.
However, when one decreases the regularization parameter to a small value for obtaining a more accurate solution, the DyKL would encounter the difficulties of numerical instabilities (due to loss of accuracy involving overflow/underflow operations) and slow convergence speed, just as Sinkhorn's algorithm \cite{s1967diagonal} employed in \cite{c2013sinkhorn} for approximately solving the classical 2-marginal OT problem. Though the former difficulty can partially be alleviated by some stabilization techniques (e.g., applying the \textit{log-sum-exp} operation \cite[Section 4.4]{pc2019computational}) at the expense of losing some computational efficiency, the latter difficulty of slow convergence, however, is unavoidable when the regularization parameter is small, as clearly observed from our numerical experiments in Section \ref{secnum}. In addition, we are not aware of fast algorithms that are specifically designed for solving the more general problem \eqref{cmotproblem}.

In this paper, we develop an implementable \underline{i}nexact \underline{e}ntropic \underline{p}roximal \underline{p}oint \underline{a}lgorithm (iEPPA) for solving problem \eqref{cmotproblem}.
Our iEPPA falls into the family of Bregman-distance-based PPA \cite{cz1992proximal,ct1993convergence,e1993nonlinear,e1998approximate} and the family of $\phi$-divergence-based PPA \cite{a1995interior,e1990multiplicative,ist1994entropy,i1995convergence,t1992entropic,t1997convergence}, both of which have been widely studied in the literature,
\blue{especially in the 1990's starting from the paper \cite{cz1992proximal}}.
However, we should point out that we have made an essential change to the algorithm by introducing a more practical stopping condition \eqref{EPPAcond1} for solving the subproblems. Therefore, existing convergence results may not be applicable and the convergence analysis has to be re-established for our iEPPA; see Theorem \ref{thmconverEPPA}. Moreover, as a byproduct, we actually develop a unified inexact framework for EPPA including Teboulle's framework \cite{t1997convergence} and Eckstein's framework \cite{e1998approximate} as special cases. This makes our iEPPA more flexible. To solve the subproblem \eqref{EPPAsubpro}, we first derive its dual problem and characterize the properties of its optimal solutions in Section \ref{secBCD}.
We then apply a block coordinate descent (BCD) method to solve the resulting dual problem and establish the linear convergence by revisiting some classical results for the BCD method in \cite{lt1992convergence,lt1993convergence,tseng1993dual}. We also show how the subproblems in the BCD method can be solved efficiently \blue{under Assumption \ref{assumption-nonoverlap}}. In particular, no stabilization technique is needed for the BCD updates since our iEPPA does not require a small proximal parameter in each iteration. This is indeed a key advantage of our iEPPA over the popular entropic regularization approach in \cite{bccnp2015iterative}. Recently, a similar algorithmic framework studied by Eckstein \cite{e1998approximate} was also adapted in \cite{xie2018fast} for solving the classical OT problem with encouraging numerical performance. However, the algorithm there was developed under a rather stringent inexact condition, which is nontrivial to verify or implement in practice.

The contributions of this paper are summarized as follows.

\begin{itemize}[leftmargin=1.05cm]
\item[{\bf 1.}] We introduce a class of specially structured LP problems \eqref{cmotproblem}, which covers some important existing problems and has favorable modeling capability. For example, it is able to formulate a tomography problem \textit{without} using a high-order tensor. This is in contrast to \cite{abraham2017tomographic,bergounioux2018variational} where a high-order (equals to two plus the number of projection directions) tensor is used to model a 2D tomography problem, and consequently the resulting problem is extremely large-scale and  prohibitively expensive to solve in terms of both memory consumption and computational cost. In addition, the third-order tensor model \eqref{cmotproblem} and the subsequent algorithms can naturally be extended to higher-order cases if needed.

\item[{\bf 2.}] We develop an efficient iEPPA combined with a dual BCD method, namely, iEPPA+BCD, to solve the proposed structured LP problem \eqref{cmotproblem}. It has the important strength of being able to faithfully solve the original problem \textit{without} requiring the proximal parameter to be very small. As a result,  when solving the CMOT problem \eqref{cmotproblem3d}, it can bypass the inherent numerical instabilities that often plague the entropic regularization approach. While our iEPPA+BCD framework is not completely new but a novel combination of existing algorithms in the optimization literature, we have nevertheless introduced an essential modification to make the algorithm practically implementable by proposing a computationally checkable stopping condition for finding a sufficiently accurate approximate solution of the subproblem in each iEPPA iteration to ensure the convergence of the overall algorithm.

\item[{\bf 3.}] We conduct rigorous numerical experiments to illustrate the efficiency of our iEPPA+BCD framework for solving the CMOT problem \eqref{cmotproblem3d}, in comparison to the (stabilized) DyKL and the powerful commercial solver Gurobi.
    Experiments on the discrete tomography problem also show the favorable modeling power of our model. 

\end{itemize}



The rest of this paper is organized as follows.
The iEPPA for solving problem \eqref{cmotproblem} and its convergence results are described in Section \ref{secEPPA}. The dual BCD method for solving the subproblem and its convergence analysis are presented in Section \ref{secBCD}. Moreover, the details on the implementable verification of our new inexact condition
is also included in Section \ref{secBCD}. Extensive numerical results are reported in Section \ref{secnum}, with some concluding remarks given in Section \ref{seccon}.


\vspace{2mm}
\paragraph{\textbf{Notation and Preliminaries}} The elements of a third-order tensor $X\in\R^{n_1\times n_2 \times n_3}$ are denoted as $X_{rst}$ where $1\leq r\leq n_1,\; 1\leq s\leq n_2,\; 1\leq t\leq n_3$. For any tensors $X$, $Y\in \R^{n_1\times n_2 \times n_3}$, we define their inner product as $\inprod{X}{Y} := {\textstyle\sum_{r = 1}^{n_1}\sum_{s = 1}^{n_2}\sum_{t=1}^{n_3}} \,X_{rst}\,Y_{rst}$. The Frobenius norm of $X$ is defined by $\norm{X}_F := \sqrt{\inprod{X}{X}}$. For any $X$, $Y\in \R^{n_1\times n_2 \times n_3}$, the Hadamard product of $X$ and $Y$ is defined by $(X\circ Y)_{rst} :=  X_{rst}\,Y_{rst}$ for any $1\leq r\leq n_1$, $1\leq s\leq n_2$, $1\leq t\leq n_3$. Similarly, we use ``$./$" to denote the element-wise division operator.
We use ``$\otimes$" to denote the tensor product of vectors. Specifically, let $\bm{u}^{(i)}\in \R^{n_i}$ ($i=1, \,2, \,3$) be three arbitrary column vectors. Their tensor product is denoted by $\bm{u}^{(1)}\otimes \bm{u}^{(2)}\otimes \bm{u}^{(3)} \in \R^{n_1\times n_2 \times n_3}$ whose elements are given by $(\bm{u}^{(1)}\otimes \bm{u}^{(2)}\otimes \bm{u}^{(3)})_{rst}: = u^{(1)}_ru^{(2)}_s u^{(3)}_t$ for any $1\leq r\leq n_1$, $1\leq s\leq n_2$, $1\leq t\leq n_3$.

Let $\mathbb{E}$ be a finitely dimensional real Euclidean space equipped with an inner product $\langle\cdot, \cdot\rangle$ and its induced norm $\|\cdot\|$. For an extended-real-valued function $f: \mathbb{E} \rightarrow [-\infty,\infty]$, we say that it is \textit{proper} if $f(\bm{x}) > -\infty$ for all $\bm{x} \in \mathbb{E}$ and its domain ${\rm dom}\,f:=\{\bm{x} \in \mathbb{E} : f(\bm{x}) < \infty\}$ is nonempty. A proper function $f$ is said to be closed if it is lower semicontinuous. Assume that $f: \mathbb{E} \rightarrow (-\infty, \infty]$ is a proper closed convex function. For a given $\nu \geq 0$, the $\nu$-subdifferential of $f$ at $\bm{x}\in{\rm dom}\,f$ is defined by $\partial_\nu f(\bm{x}):=\{\bm{d}\in\mathbb{E}: f(\bm{y}) \geq f(\bm{x}) + \langle \bm{d}, \,\bm{y}-\bm{x}  \rangle -\nu, ~\forall\,\bm{y}\in\mathbb{E}\}$ and when $\nu=0$, $\partial_\nu f$ is simply denoted by $\partial f$, which is referred to as the subdifferential of $f$. The conjugate function of $f$ is $f^*: \mathbb{E} \rightarrow (-\infty,\infty]$ defined by $f^*(\bm{y}):=\sup\left\{\langle \bm{y},\,\bm{x}\rangle-f(\bm{x}) : \bm{x}\in\mathbb{E}\right\}$. For any $\bm{x}$, $\bm{y}\in\mathbb{E}$, it follows from \cite[Theorem 23.5]{r1970convex} that
\begin{equation}\label{subeqv}
\bm{y} \in \partial f(\bm{x}) ~~\Longleftrightarrow~~ \bm{x} \in \partial f^*(\bm{y}).
\end{equation}
Moreover, we call a proper closed convex function $f$ essentially smooth if (i) the interior of $\mathrm{dom}\,f$, denoted by $\mathrm{int}\,\mathrm{dom}\,f$, is not empty; (ii) $f$ is differentiable on $\mathrm{int}\,\mathrm{dom}\,f$; (iii) $\|\nabla f(x_k)\|\to\infty$ for every sequence $\{x_k\}$ in $\mathrm{int}\,\mathrm{dom}\,f$ converging to a boundary point of $\mathrm{int}\,\mathrm{dom}\,f$; see \cite[page 251]{r1970convex}.


Finally, we make a blanket assumption throughout this paper.

\begin{assumption}\label{assumpfeas}
The feasible set $\Omega$ is bounded and $\Omega\cap\mathbb{R}^{n_1 \times n_2 \times n_3}_{++}$ is nonempty.
\end{assumption}

This assumption ensures the well-definedness of problem \eqref{cmotproblem} and our method developed in the next section. The boundedness assumption can be satisfied if, for instance, $U\in \R_{+}^{n_1\times n_2\times n_3}$ or the constraints are given as in \eqref{cmotproblem3d}.

\section{An implementable inexact entropic proximal point algorithm}\label{secEPPA}

In this section, we develop an implementable \underline{i}nexact \underline{e}ntropic \underline{p}roximal \underline{p}oint \underline{a}lgorithm (iEPPA) for solving problem \eqref{cmotproblem}. To describe the iterates of the iEPPA, we first rewrite problem \eqref{cmotproblem} as follows:
\begin{equation}\label{cmotproblem3dre}
\min\limits_{X}~\delta_{\Omega^{\circ}}(X) + \langle {C}, \,X\rangle, \quad \mathrm{s.t.} \quad X\geq0,
\end{equation}
where $\delta_{\Omega^{\circ}}(\cdot)$ is the indicator function of the set $\Omega^{\circ}$ defined as
\begin{equation*}
\Omega^{\circ}:=\big\{X\in\mathbb{R}^{n_1 \times n_2 \times n_3} ~:~ \cA^{(i)}(X) = \bm{b}^{(i)}, ~i = 1, \dots, N, ~X \leq U \big\}.
\end{equation*}
Obviously, the set $\Omega^{\circ}$ is formed by removing the non-negative constraint on $X$ from the set $\Omega$ and hence $\Omega\subseteq\Omega^{\circ}$. We also introduce the Boltzmann-Shannon entropy function $\phi(X)=\sum_{rst}X_{rst}\log X_{rst}-X_{rst}$ (with the convention that $0\log 0 = 0$) and its conjugate function $\phi^*(Y)=\sum_{rst}\exp(Y_{rst})$. Then, the Bregman distance \cite{b1967relaxation} with $\phi$ as the kernel function, is defined as
\begin{equation*}
\mathcal{D}_{\phi}(X, \,Y) := \phi(X) - \phi(Y) - \langle \nabla\phi(Y), \,X-Y \rangle, \quad \forall\,X\in\mathbb{R}_+^{n_1 \times n_2 \times n_3},~ Y\in\mathbb{R}_{++}^{n_1 \times n_2 \times n_3}.
\end{equation*}
It is easy to see that $D_{\phi}(X, \,Y)\geq0$ and the equality holds if and only if $X=Y$. Then, the iEPPA for solving \eqref{cmotproblem3dre} (hence \eqref{cmotproblem}) is presented as Algorithm \ref{algEPPA}.

\begin{algorithm}[h]
\caption{An implementable inexact entropic proximal point algorithm (iEPPA) for solving \eqref{cmotproblem3dre}}\label{algEPPA}
\textbf{Input:} Let $\{\varepsilon_k\}_{k=0}^{\infty}$, $\{\nu_k\}_{k=0}^{\infty}$, $\{\eta_k\}_{k=0}^{\infty}$ and $\{\mu_k\}_{k=0}^{\infty}$ be four sequences of nonnegative scalars. Choose $X^0=\widetilde{X}^{0}\in\mathbb{R}_{++}^{n_1 \times n_2 \times n_3}$ arbitrarily. Set $k=0$.  \\
\textbf{while} the termination criterion is not met, \textbf{do} \vspace{-2mm}
\begin{itemize}[leftmargin=2cm]
\item[\textbf{Step 1}.]
    Find a pair $(X^{k+1},\,\widetilde{X}^{k+1})$ by approximately solving the following problem
    \begin{equation}\label{EPPAsubpro}
    \min\limits_{X}~
    \delta_{\Omega^{\circ}}(X) + \langle C, \,X\rangle + \varepsilon_k\,\mathcal{D}_{\phi}(X, \,X^k),
    \end{equation}
    such that $X^{k+1}\in\mathbb{R}_{++}^{n_1 \times n_2 \times n_3}$, $\widetilde{X}^{k+1}\in\Omega$ and
    \begin{equation}\label{EPPAcond1}
    \begin{aligned}
    &\Delta^{k} \in \partial_{\nu_k} \delta_{\Omega^{\circ}}(\widetilde{X}^{k+1}) + C + \varepsilon_k\,\big(\nabla\phi(X^{k+1}) - \nabla\phi(X^{k})\big) \\
    &\quad \mathrm{with}~~\|\Delta^{k}\|_F \leq \eta_k, ~~\mathcal{D}_{\phi}(\widetilde{X}^{k+1}, \,X^{k+1}) \leq \mu_k.
    \end{aligned}
    \end{equation}

\item [\textbf{Step 2}.] Set $k = k+1$ and go to \textbf{Step 1}. \vspace{-1.5mm}
\end{itemize}
\textbf{end while}  \\
\textbf{Output}: $(X^{k},\,\widetilde{X}^{k})$ \vspace{0.5mm}
\end{algorithm}

The reader may have observed that the iEPPA in Algorithm \ref{algEPPA} basically solves the original problem \eqref{cmotproblem3dre} (hence \eqref{cmotproblem}) via approximately solving a sequence of subproblems \eqref{EPPAsubpro} each involving a special entropic Bregman proximal term. Since $\mathrm{dom}\,\phi = \mathbb{R}_{+}^{n_1 \times n_2 \times n_3}$, the constraint $X\geq0$ can be removed in \eqref{EPPAsubpro}. Moreover, the Boltzmann-Shannon entropy function $\phi(X)=\sum_{rst}X_{rst}\log X_{rst} - X_{rst}$ is essentially smooth and strictly convex on $\mathbb{R}_{+}^{n_1 \times n_2 \times n_3}$. This together with Assumption \ref{assumpfeas} ensures that each subproblem \eqref{EPPAsubpro} is well-defined in the sense that its optimal solution (denoted by $X^{k,*}$) uniquely exists and lies in $\mathbb{R}_{++}^{n_1 \times n_2 \times n_3}$. Indeed, since $\Omega^{\circ}\cap\mathrm{dom}\,\phi=\Omega$ is bounded, the objective function in subproblem \eqref{EPPAsubpro} is then level-bounded. Thus, a solution exists \cite[Theorem 1.9]{rw1998variational} and must be unique since $\phi$ is strictly convex. The essential smoothness of $\phi$ and Assumption \ref{assumpfeas} further imply that $X^{k,*}$ can only lie in $\mathbb{R}_{++}^{n_1 \times n_2 \times n_3}$. Notice that our inexact condition \eqref{EPPAcond1} always holds when $X^{k+1}=\widetilde{X}^{k+1}=X^{k,*}$ and hence it is achievable.

The inexact condition \eqref{EPPAcond1} is rather general to cover some existing inexact conditions, and more importantly, it makes our iEPPA more practical for solving problem \eqref{cmotproblem3dre} (hence \eqref{cmotproblem}). When $\nu_k\equiv\eta_k\equiv\mu_k\equiv0$, $X^{k+1}$ (equals to $\widetilde{X}^{k+1}$) must be the exact optimal solution of subproblem \eqref{EPPAsubpro}. In this case, our exact version of iEPPA is indeed a special case of the classical exact generalized PPA (such as the $\phi$-divergence-based PPA \cite{a1995interior,e1990multiplicative,ist1994entropy,i1995convergence,t1992entropic} and the Bregman-distance-based PPA
\cite{cz1992proximal,ct1993convergence,e1993nonlinear}). When $\eta_k\equiv\mu_k\equiv0$, condition \eqref{EPPAcond1} reduces to
\begin{equation}\label{EPPAcond1-Te}
0 \in \partial_{\nu_k} \delta_{\Omega^{\circ}}(X^{k+1}) + C + \varepsilon_k\big(\nabla\phi(X^{k+1}) - \nabla\phi(X^{k})\big),
\end{equation}
which is considered by Teboulle \cite{t1997convergence} in the $\phi$-divergence-based PPA that allows the approximate computations of the subdifferential of $\delta_{\Omega^{\circ}}$ at $X^{k+1}$, provided $X^{k+1}\in\Omega^\circ$. When $\nu_k\equiv\mu_k\equiv0$, condition \eqref{EPPAcond1} reduces to
\begin{equation}\label{EPPAcond1-Ec}
\Delta^{k} \in \partial\delta_{\Omega^{\circ}}(X^{k+1}) + C + \varepsilon_k\big(\nabla\phi(X^{k+1}) - \nabla\phi(X^{k})\big)
~~\mathrm{with}~~\|\Delta^{k}\|_F \leq \eta_k,
\end{equation}
which is considered by Eckstein \cite{e1998approximate} in the Bregman-distance-based PPA and is typically easier to check than the $\nu$-subdifferential-based condition \eqref{EPPAcond1-Te}. But again, it requires $X^{k+1}$ to be in $\Omega^\circ$. The inexact algorithmic framework based on condition \eqref{EPPAcond1-Ec} has also been adapted in \cite{xie2018fast} for solving the classical 2-marginal OT problem. However, we should mention that {\it neither} Teboulle's inexact condition \eqref{EPPAcond1-Te} {\it nor} Eckstein's inexact condition \eqref{EPPAcond1-Ec} is easy to implement for solving the subproblem with the complicated constraint that $X\in\Omega^{\circ}$. Because it is nontrivial to find a point $X^{k+1}$ that {\em simultaneously} satisfies $X^{k+1}\in\Omega^{\circ}$ (required by the nonemptyness of $\partial_{\nu_k} \delta_{\Omega^{\circ}}(X^{k+1})$ or $\partial \delta_{\Omega^{\circ}}(X^{k+1})$) and $X^{k+1}\in\mathbb{R}_{++}^{n_1 \times n_2 \times n_3}$ (required by the essentially smoothness of $\phi$). This inadequacy thus motivated us to further relax conditions \eqref{EPPAcond1-Te} and \eqref{EPPAcond1-Ec} to condition \eqref{EPPAcond1}, in which $\partial_{\nu_k} \delta_{\Omega^{\circ}}$ and $\nabla \phi$ are allowed to be computed at two slightly different points, respectively. We shall show later in subsection \ref{sec-verif} that the verification of our inexact condition \eqref{EPPAcond1} is more practically implementable.


We next establish the convergence of our iEPPA in Algorithm \ref{algEPPA}. Our analysis is inspired by several existing works (see, for example, \cite{e1998approximate,t1997convergence}), but is more involved due to the flexible inexact condition \eqref{EPPAcond1}. We shall start with some elementary preliminaries. It is known from \cite[Section 6.1]{cl1981iterative} that the Boltzmann-Shannon entropy function $\phi$ has many elegant properties as a Bregman function (see \cite[Definition 2.1]{cl1981iterative}). We point out three of them below that are useful in our subsequential analysis. More details on the Bregman function can be found in \cite[Section 4]{bb1997legendre}.

\begin{property}\label{phipros}
The following properties hold for $\phi(X)=\sum_{rst}X_{rst}\log X_{rst} - X_{rst}$.
\begin{itemize}
\item[{\rm (i)}] For any $X\in\mathbb{R}_{+}^{n_1 \times n_2 \times n_3}$, $\mathcal{D}_{\phi}(X,\,\cdot)$ is level-bounded.

\item[{\rm (ii)}] If $\{Y^k\}\subseteq\mathbb{R}_{++}^{n_1 \times n_2 \times n_3}$ converges to some $Y^{*}\in\mathbb{R}_{+}^{n_1 \times n_2 \times n_3}$, then $\mathcal{D}_{\phi}(Y^{*},\,Y^k)\to0$.

\item[{\rm (iii)}] \textbf{(Convergence consistency)} If $\{X^k\}\subseteq\mathbb{R}_{+}^{n_1 \times n_2 \times n_3}$ and $\{Y^k\}\subseteq\mathbb{R}_{++}^{n_1 \times n_2 \times n_3}$ are two sequences such that $\{X^k\}$ is bounded, $Y^k\to Y^{*}$ and $\mathcal{D}_{\phi}(X^k,\,Y^k)\to0$, then $X^k\to Y^{*}$.

\end{itemize}
\end{property}

We also recall two well-known results.

\begin{lemma}[{\bf Three points identity}\label{threeIdenity}
{\cite[Lemma 3.1]{ct1993convergence}}]\label{lem3eq}
For any $X\in\mathbb{R}^{n_1 \times n_2 \times n_3}_+$ and $Y,\,Z\in\mathbb{R}^{n_1 \times n_2 \times n_3}_{++}$, the following identity holds:
\begin{equation*}
\langle \nabla\phi(Y)-\nabla\phi(Z),\,X-Y \rangle = \mathcal{D}_{\phi}(X,\,Z)-\mathcal{D}_{\phi}(X,\,Y)-\mathcal{D}_{\phi}(Y,\,Z).
\end{equation*}
\end{lemma}

\begin{lemma}[{\cite[Section 2.2]{p1987introduction}}]\label{lemseqcon}
Suppose that $\{a_k\}_{k=0}^{\infty}\subseteq\mathbb{R}$ and $\{\gamma_k\}_{k=0}^{\infty}\subseteq\mathbb{R}$ are two sequences such that $\{a_k\}$ is bounded from below, $\sum_{k=0}^{\infty} \gamma_k < \infty$, and $a_{k+1} \leq a_{k} + \gamma_k$ holds for all $k$. Then, $\{a_k\}$ is convergent.
\end{lemma}

We are now ready to give the main convergence result.

\begin{theorem}[\textbf{Convergence of the iEPPA}]\label{thmconverEPPA}
Suppose that Assumption \ref{assumpfeas} holds and $\{\varepsilon_k\}_{k=0}^{\infty}$, $\{\nu_k\}_{k=0}^{\infty}$, $\{\eta_k\}_{k=0}^{\infty}$, $\{\mu_k\}_{k=0}^{\infty}$ are four sequences of nonnegative scalars. Let $\{X^{k}\}$ and $\{\widetilde{X}^{k}\}$ be the sequences generated by the iEPPA in Algorithm \ref{algEPPA}. If $0<\underline{\varepsilon}\leq\varepsilon_k\leq\bar{\varepsilon}<\infty$, $\sum\nu_k<\infty$, $\sum\eta_k<\infty$ and $\sum\mu_k<\infty$, then $\{X^{k}\}$ and $\{\widetilde{X}^{k}\}$ converge to a same optimal solution of problem \eqref{cmotproblem3dre} (hence problem \eqref{cmotproblem}).
\end{theorem}
\begin{proof}
First, from condition \eqref{EPPAcond1}, there exists a $D^{k+1}\in\partial_{\nu_k} \delta_{\Omega^{\circ}}(\widetilde{X}^{k+1})$ such that
\begin{equation*}
\Delta^{k} = D^{k+1} + C + \varepsilon_k\big(\nabla\phi(X^{k+1}) - \nabla\phi(X^{k})\big).
\end{equation*}
Then, for any $P\in\Omega\subseteq\Omega^{\circ}$, we see that
\begin{equation*}
0 \geq \langle D^{k+1}, \,P-\widetilde{X}^{k+1} \rangle - \nu_k
= \langle \Delta^{k} - C - \varepsilon_k\big(\nabla\phi(X^{k+1}) - \nabla\phi(X^{k})\big), \,P-\widetilde{X}^{k+1} \rangle - \nu_k,
\end{equation*}
which implies that
\begin{equation}\label{ineq1}
\langle C, \,\widetilde{X}^{k+1}\rangle \leq \langle C, \,P\rangle + \varepsilon_k\langle \,\nabla\phi(X^{k+1}) - \nabla\phi(X^{k}), \,P-\widetilde{X}^{k+1} \,\rangle + \langle \Delta^{k}, \,\widetilde{X}^{k+1}-P \rangle + \nu_k.
\end{equation}
Note that
\begin{equation}\label{ineq2}
\begin{aligned}
&\quad \;\langle \,\nabla\phi(X^{k+1}) - \nabla\phi(X^{k}), \,P-\widetilde{X}^{k+1} \,\rangle \\
&=\langle\,\nabla\phi(X^{k+1}) - \nabla\phi(X^{k}), \,P-X^{k+1} \,\rangle-\langle \,\nabla\phi(X^{k+1}) - \nabla\phi(X^{k}), \,\widetilde{X}^{k+1}-X^{k+1} \,\rangle \\
&=\mathcal{D}_{\phi}(P,\,X^k)-\mathcal{D}_{\phi}(P,\,X^{k+1})
-\mathcal{D}_{\phi}(X^{k+1},\,X^k)
-\big( \mathcal{D}_{\phi}(\widetilde{X}^{k+1},\,X^k)
-\mathcal{D}_{\phi}(\widetilde{X}^{k+1},\,X^{k+1})
-\mathcal{D}_{\phi}(X^{k+1},\,X^k) \big) \\
&=\mathcal{D}_{\phi}(P,\,X^k)-\mathcal{D}_{\phi}(P,\,X^{k+1}) -\mathcal{D}_{\phi}(\widetilde{X}^{k+1},\,X^k)
+\mathcal{D}_{\phi}(\widetilde{X}^{k+1},\,X^{k+1}) \\
&\leq \mathcal{D}_{\phi}(P,\,X^k)
-\mathcal{D}_{\phi}(P,\,X^{k+1})
-\mathcal{D}_{\phi}(\widetilde{X}^{k+1},\,X^k) + \mu_k,
\end{aligned}
\end{equation}
where the second equality follows from the three points identity in Lemma \ref{threeIdenity}. Moreover, since $\widetilde{X}^{k+1}\in\Omega$ and $P\in\Omega$, then $\langle \Delta^{k}, \,\widetilde{X}^{k+1}-P \rangle\leq\|\widetilde{X}^{k+1}-P\|_F\|\Delta^{k}\|_F\leq2\rho\eta_k$, where the last inequality follows from the boundedness of $\Omega$ (by Assumption \ref{assumpfeas}) and hence there exists a $\rho>0$ such that $\|X\|_F\leq\rho$ for all $X\in\Omega$. Combining this with \eqref{ineq1} and \eqref{ineq2}, we have
\begin{equation}\label{desineq-U}
\langle C, \,\widetilde{X}^{k+1}\rangle \leq \langle C, \,P\rangle + \varepsilon_k\big( \mathcal{D}_{\phi}(P,\,X^k)
-\mathcal{D}_{\phi}(P,\,X^{k+1})
-\mathcal{D}_{\phi}(\widetilde{X}^{k+1},\,X^k)\big) + \varepsilon_k\mu_k + 2\rho\eta_k + \nu_k, ~~ \forall\,P\in\Omega.
\end{equation}

Now, set $P=\widetilde{X}^k$ in \eqref{desineq-U}, we see that
\begin{equation}\label{desineq-Xk}
\begin{aligned}
\langle C, \,\widetilde{X}^{k+1}\rangle
&\leq\langle C, \,\widetilde{X}^k\rangle + \varepsilon_k\big( \mathcal{D}_{\phi}(\widetilde{X}^k,\,X^k)
-\mathcal{D}_{\phi}(\widetilde{X}^k,\,X^{k+1}) -\mathcal{D}_{\phi}(\widetilde{X}^{k+1},\,X^k) \big)
+ \varepsilon_k\mu_k + 2\rho\eta_k + \nu_k\\
&\leq \langle C, \,\widetilde{X}^k\rangle - \varepsilon_k\big( \mathcal{D}_{\phi}(\widetilde{X}^k,\,X^{k+1}) +\mathcal{D}_{\phi}(\widetilde{X}^{k+1},\,X^k) \big) + \varepsilon_k(\mu_{k-1}+\mu_k) + 2\rho\eta_k + \nu_k \\
&\leq \langle C, \,\widetilde{X}^k\rangle + \varepsilon_k(\mu_{k-1}+\mu_k) + 2\rho\eta_k + \nu_k.
\end{aligned}
\end{equation}
Note that $\{\langle C, \,\widetilde{X}^k\rangle\}$ is bounded below since $\widetilde{X}^k$ is in the compact set $\Omega$ for all $k$. Then, since $\varepsilon_k$ is nonnegative and bounded from above, $\sum\nu_k<\infty$, $\sum\eta_k<\infty$ and $\sum\mu_k<\infty$, it follows from \eqref{desineq-Xk} and Lemma \ref{lemseqcon} that $\{\langle C, \,\widetilde{X}^k\rangle\}$ is convergent. Also, we see from \eqref{desineq-Xk} that
\begin{equation*}
\varepsilon_k\big( \mathcal{D}_{\phi}(\widetilde{X}^k,\,X^{k+1}) +\mathcal{D}_{\phi}(\widetilde{X}^{k+1},\,X^k) \big) \leq \langle C, \,\widetilde{X}^k\rangle - \langle C, \,\widetilde{X}^{k+1}\rangle + \varepsilon_k(\mu_{k-1}+\mu_k) + 2\rho\eta_k + \nu_k.
\end{equation*}
From this, together with $\varepsilon_k\geq\underline{\varepsilon}>0$, $\nu_k\to0$, $\eta_k\to0$, $\mu_k\to0$ and the fact that $\{\langle C, \,\widetilde{X}^k\rangle\}$ is convergent, we get that
\begin{equation*}
\mathcal{D}_{\phi}(\widetilde{X}^k,\,X^{k+1}) \to 0 \quad \mathrm{and} \quad \mathcal{D}_{\phi}(\widetilde{X}^{k+1},\,X^k) \to 0.
\end{equation*}

Next, let $X^*$ be an arbitrary optimal solution of \eqref{cmotproblem3dre} (hence \eqref{cmotproblem}). Obviously, $\langle C,\,X^*\rangle \leq \langle C, \,\widetilde{X}^{k+1}\rangle$ for all $k$ since $\widetilde{X}^{k+1}\in\Omega$. By setting $P=X^*$ in \eqref{desineq-U}, dividing the resulting inequality by $\varepsilon_k$ and rearranging terms, we see that
\begin{equation}\label{desineq-Xstar}
\begin{aligned}
0&\leq\mathcal{D}_{\phi}(X^*,\,X^{k+1}) \\
&\leq \mathcal{D}_{\phi}(X^*,\,X^k) + \varepsilon_k^{-1}\big(\langle C, \,X^*\rangle-\langle C, \,\widetilde{X}^{k+1}\rangle\big) -\mathcal{D}_{\phi}(\widetilde{X}^{k+1},\,X^k) + \mu_k + \varepsilon_k^{-1}(2\rho\eta_k + \nu_k) \\
&\leq \mathcal{D}_{\phi}(X^*,\,X^k) + \mu_k + \varepsilon_k^{-1}(2\rho\eta_k + \nu_k).
\end{aligned}
\end{equation}
Thus, we can conclude from the above inequality and Lemma \ref{lemseqcon} that $\{\mathcal{D}_{\phi}(X^*,\,X^k)\}$ is convergent. On the other hand, since $\{\widetilde{X}^k\}$ is bounded (due to $\widetilde{X}^k\in\Omega$), it has at least one cluster point. Suppose that $\widetilde{X}^{\infty}$ is a cluster point and $\{\widetilde{X}^{k_i}\}$ is a convergent subsequence such that $\lim_{i\to\infty} \widetilde{X}^{k_i} = \widetilde{X}^{\infty}$. Then, by using \eqref{desineq-U} with $P=X^*$ again, we have for all $k_i$,
\begin{equation*}
\begin{aligned}
\langle C, \,\widetilde{X}^{k_i}\rangle
&\leq \langle C, \,X^*\rangle + \varepsilon_{k_i-1}\big( \mathcal{D}_{\phi}(X^*,\,X^{k_i-1})
-\mathcal{D}_{\phi}(X^*,\,X^{k_i}) -\mathcal{D}_{\phi}(\widetilde{X}^{k_i},\,X^{k_i-1}) \big) \\
&\qquad + \varepsilon_{k_i-1}\mu_{k_i-1} + 2\rho\eta_{k_i-1} + \nu_{k_i-1}  \\
&\leq \langle C, \,X^*\rangle + \varepsilon_{k_i-1}\big( \mathcal{D}_{\phi}(X^*,\,X^{k_i-1})
-\mathcal{D}_{\phi}(X^*,\,X^{k_i}) \big)
+ \varepsilon_{k_i-1}\mu_{k_i-1} + 2\rho\eta_{k_i-1} + \nu_{k_i-1}.
\end{aligned}
\end{equation*}
Then, passing to the limit and recalling that $\{\mathcal{D}_{\phi}(X^*,\,X^k)\}$ is convergent, $0<\underline{\varepsilon}\leq\varepsilon_k\leq\bar{\varepsilon}<\infty$, $\nu_k\to0$, $\eta_k\to0$, $\mu_k\to0$, we obtain that
\begin{equation*}
\langle C, \,\widetilde{X}^{\infty}\rangle \leq \langle C, \,X^*\rangle.
\end{equation*}
Note that $\widetilde{X}^{\infty}\in\Omega$ since $\Omega$ is closed. Thus, $\widetilde{X}^{\infty}$ is an optimal solution of \eqref{cmotproblem3dre} (hence \eqref{cmotproblem}).

In addition, from Property \ref{phipros}(i) and the fact that $\{\mathcal{D}_{\phi}(X^*,\,X^k)\}$ is convergent, we can conclude that $\{X^k\}$ must be bounded and hence it has at least one cluster point. Suppose that $X^{\infty}$ is a cluster point and $\{X^{k_j}\}$ is a convergent subsequence such that $\lim_{j\to\infty} X^{k_j} = X^{\infty}$. Then, from $\mathcal{D}_{\phi}(\widetilde{X}^{k_j}, \,X^{k_j})\leq\mu_{k_j-1}\to0$, the boundedness of $\{\widetilde{X}^{k_j}\}$ and Property \ref{phipros}(iii), we have that $\lim_{j\to\infty}\widetilde{X}^{k_j}= X^{\infty}$. Therefore, from what we have proved in the last paragraph, $X^{\infty}$ is an optimal solution of \eqref{cmotproblem3dre} (hence \eqref{cmotproblem}), and moreover, by using \eqref{desineq-Xstar} with $X^*$ replaced by $X^{\infty}$, we can conclude that $\{\mathcal{D}_{\phi}(X^{\infty},\,X^k)\}$ is convergent. On the other hand, it follows from $\lim_{j\to\infty} X^{k_j} = X^{\infty}$ and Property \ref{phipros}(ii) that $\mathcal{D}_{\phi}(X^{\infty}, \,X^{k_j})\to0$. Consequently, $\{\mathcal{D}_{\phi}(X^{\infty},\,X^k)\}$ must converge to zero. Now, let $\widehat{X}^{\infty}$ be any cluster point of $\{X^k\}$ with a subsequence $\{X^{k'_j}\}$ such that $X^{k'_j}\to\widehat{X}^{\infty}$. Since $\mathcal{D}_{\phi}(X^{\infty},\,X^k)\to0$, we have $\mathcal{D}_{\phi}(X^{\infty},\,X^{k'_j})\to0$. Using Property \ref{phipros}(iii) again, we see that $X^{\infty}=\widehat{X}^{\infty}$. Since $\widehat{X}^{\infty}$ is arbitrary, we can conclude that $\lim_{k\to\infty}X^k=X^{\infty}$. This, together with the boundedness of $\{\widetilde{X}^k\}$, $\mathcal{D}_{\phi}(\widetilde{X}^{k}, \,X^{k})\to0$ and Property \ref{phipros}(iii), implies that $\{\widetilde{X}^k\}$ also converges to $X^{\infty}$. We then complete the proof.
\end{proof}

From Theorem \ref{thmconverEPPA}, we see that the convergence of our iEPPA can be easily guaranteed with proper choices of $\{\varepsilon_k\}$, $\{\nu_k\}$, $\{\eta_k\}$ and $\{\mu_k\}$. To make our iEPPA truly implementable, we \blue{will illustrate in the next section} how to efficiently solve the subproblem \eqref{EPPAsubpro} to find a pair $(X^{k},\,\widetilde{X}^{k})$ satisfying condition \eqref{EPPAcond1} at each iteration (see {\bf Step} 1 in Algorithm \ref{algEPPA}).


\section{A dual block coordinate descent method for solving \eqref{EPPAsubpro}}\label{secBCD}

In this section, we present an efficient method for solving the subproblem \eqref{EPPAsubpro}. Specifically, we first derive the dual problem of \eqref{EPPAsubpro}, which is conceivably more tractable, and then apply a block coordinate descent (BCD) method for solving it. Note that the subproblem \eqref{EPPAsubpro} has the same form as the entropic regularized counterpart of problem \eqref{cmotproblem}. Thus, one can also follow \cite{bccnp2015iterative} to apply the DyKL for solving it. However, our numerical comparisons have shown that the dual BCD method is more efficient than the DyKL for solving \eqref{EPPAsubpro} with a fixed $\varepsilon_k$ and hence it can be of independent interest for solving an entropic regularized problem in form of \eqref{EPPAsubpro}.\footnote{In this paper, we omit numerical comparisons between the dual BCD and DyKL to save space, and refer the interested reader to our early arXiv version (arXiv:2011.14312v2).}

For notational simplicity, we drop the index $k$ and consider the following generic problem with given $S\in\mathbb{R}^{n_1 \times n_2 \times n_3}_{++}$ and $\varepsilon>0$:
\begin{equation}\label{subcmot3d}
\begin{aligned}
&\min\limits_{X\in\mathbb{R}^{n_1 \times n_2 \times n_3}}\,\,\,\langle {C}, \,X\rangle + \varepsilon\,\mathcal{D}_{\phi}(X, \,S) \\
&\hspace{0.8cm} \mathrm{s.t.} \hspace{0.9cm} \cA^{(i)}(X) - \bm{b}^{(i)} = 0, ~~i = 1, \dots, N, \\
&\hspace{2.2cm} X \leq U,
\end{aligned}
\end{equation}
where $\phi(X)=\sum_{rst}X_{rst}\log X_{rst}-X_{rst}$. By introducing an auxiliary variable $Z\in\mathbb{R}^{n_1 \times n_2 \times n_3}$ and substituting $\phi$ into \eqref{subcmot3d}, we can equivalently reformulate problem \eqref{subcmot3d} as
\begin{equation}\label{subcmot3dre}
\begin{aligned}
&\min\limits_{X, \,Z\in\mathbb{R}^{n_1 \times n_2 \times n_3}}~\langle M, \,X\rangle + \varepsilon\,{\textstyle\sum_{r,s,t}} X_{rst}\left(\log X_{rst} - 1\right) + \delta_+(Z) \\
&\hspace{0.95cm} \mathrm{s.t.} \hspace{1.05cm} \cA^{(i)}(X) - \bm{b}^{(i)} = 0, ~~i = 1, \dots, N, \\
&\hspace{2.5cm} X + Z = U,
\end{aligned}
\end{equation}
where $M := C - \varepsilon\log S$ and $\delta_{+}(\cdot)$ is the indicator function over the set $\big{\{}Z\in\mathbb{R}^{n_1 \times n_2 \times n_3}: Z\geq0\big{\}}$. The Lagrangian function associated with \eqref{subcmot3dre} is
\begin{equation*}
\begin{aligned}
\quad \mathcal{L}\big{(}X, Z, \bm{y}^{(1)},\dots, \bm{y}^{(N)},W \big{)}
&=\big\langle M-{\textstyle\sum_{i=1}^N}\cA^{(i,*)}\bm{y}^{(i)} -W, \,X \big\rangle + \varepsilon\,{\textstyle\sum_{r,s,t}} X_{rst}\left(\log X_{rst} - 1\right)\\
&\hspace{0.5cm}
+ \delta_+(Z) - \inprod{W}{Z} + {\textstyle\sum_{i=1}^N}\inprod{\bm{y}^{(i)}}{\bm{b}^{(i)}} + \inprod{W}{U} ,
\end{aligned}
\end{equation*}
where $ \bm{y}^{(i)}\in \R^{m_i}\,(i=1,\ldots,N)$, $W\in\mathbb{R}^{n_1 \times n_2 \times n_3}$ are Lagrangian multipliers for \eqref{subcmot3dre} and $\cA^{(i,*)}: \R^{m_i} \to \R^{n_1\times n_2\times n_3}$ is the adjoint mapping of $\cA^{(i)}$ that is defined by $\cA^{(i,*)} \bm{y}^{(i)} := \sum_{j=1}^{m_i} y^{(i)}_j A^{(i)}_j$. Then, the dual problem of \eqref{subcmot3dre} is given by
\begin{equation}\label{dualorg}
\max\limits_{\bm{y}^{(1)},\dots, \bm{y}^{(N)},W }\left\{\min\limits_{X,Z}~\mathcal{L}\big{(}X,Z, \bm{y}^{(1)},\dots, \bm{y}^{(N)},W \big{)}\right\}.
\end{equation}
Observe that
\begin{equation*}
\min_X\Big\{\Inprod{M-{\sum_{i=1}^N}\cA^{(i,*)}\bm{y}^{(i)} -W }{X}+ \varepsilon\,{\sum_{r,s,t}} X_{rst}\left(\log X_{rst} - 1\right) \Big\} = -\varepsilon \Inprod{\widetilde{M}}{\exp\Big( \varepsilon^{-1}\big( W+{\sum_{i=1}^N}\cA^{(i,*)}\bm{y}^{(i)} \big) \Big)},
\end{equation*}
where $\widetilde{M}: = \exp(-M/\varepsilon)= S\circ\exp(-C/\varepsilon)$ and
\begin{equation*}
\min\limits_{Z}\big{\{}\delta_+(Z) - \langle W, \,Z \rangle\big{\}} =
\left\{\begin{array}{ll}
0, &\mathrm{if}~~W \leq 0, \\
-\infty, &\mathrm{otherwise}.
\end{array}\right.
\end{equation*}
Here the notation $\exp(X)$ means that the exponential operation is applied to all entries of $X$. With these facts and some manipulations, problem \eqref{dualorg} is then equivalent to
\begin{eqnarray}\label{subcmot3dredual}
\min\limits_{ \bm{y}^{(1)},\dots, \bm{y}^{(N)},W} \!\left\{\!\!\!
\begin{array}{ll}
R\big{(} \bm{y}^{(1)},\dots, \bm{y}^{(N)},W\big{)}
\vspace{2mm}\\
:= \varepsilon
\big\langle \widetilde{M}, \,\exp\big( \varepsilon^{-1}\big( W+\sum_{i=1}^N\cA^{(i,*)}\bm{y}^{(i)} \big)\big)\big\rangle
- \sum_{i=1}^N\inprod{\bm{y}^{(i)}}{\bm{b}^{(i)}} - \inprod{W}{U} +  \delta_-(W)
\end{array}
\!\!\!\right\},
\end{eqnarray}
where $\delta_{-}(\cdot)$ is the indicator function over the set $\big{\{}W\in\mathbb{R}^{n_1 \times n_2 \times n_3}: W\leq0\big{\}}$. Now, we see that problem \eqref{subcmot3dredual} is a convex problem with $N+1$ blocks of variables and is conceivably more tractable than the original problem \eqref{subcmot3d}. Indeed, for this kind of problems containing several blocks of variables, it is desirable to apply the BCD method, which basically minimizes the objective $R$ with respect to $ \bm{y}^{(1)},\dots, \bm{y}^{(N)},W$ cyclically at each iteration; see Algorithm \ref{algBCD} for a detailed description.

\begin{algorithm}[h]
\caption{A dual block coordinate descent method for solving \eqref{subcmot3d}}\label{algBCD}
\textbf{Input:} Choose $(\bm{y}^{(1),0},\dots,\bm{y}^{(N),0},W^0)\in\mathrm{dom}\,R$ arbitrarily. Set $\ell=0$.  \\
\textbf{while} a termination criterion is not met, \textbf{do} \vspace{-2mm}
\begin{itemize}[leftmargin=2cm]
\item[\textbf{Step 1}.] compute
\begin{equation*}
\begin{aligned}
\bm{y}^{(i),\ell+1} &= \arg\min_{\bm{y}^{(i)}}\,R\big{(}\bm{y}^{(1),\ell+1}, \,\dots\!,\, \bm{y}^{(i-1),\ell+1},\,\bm{y}^{(i)}, \,\bm{y}^{(i+1),\ell}, \,\dots\!,\,\bm{y}^{(N),\ell},\, W^\ell\big{)}, \quad 1\leq i\leq N, \\
W^{\ell+1} &= \arg\min_{W}\,R\big{(}\bm{y}^{(1),\ell+1},\,\dots,\;\bm{y}^{(N),\ell+1},\,W\big{)}.
\end{aligned}\vspace{-2mm}
\end{equation*}

\item [\textbf{Step 2}.] Set $\ell=\ell+1$ and go to \textbf{Step 1}. \vspace{-1.5mm}
\end{itemize}
\textbf{end while}  \\
\textbf{Output}: $(\bm{y}^{(1),\ell},\dots,\bm{y}^{(N),\ell},W^{\ell})$ \vspace{0.5mm}
\end{algorithm}

We will show in the next subsection that the dual BCD in Algorithm \ref{algBCD}
is R-linearly convergent and also provides the optimal solution of problem \eqref{subcmot3d}.
Moreover, by using the nice structures imposed on $\cA^{(i)}$ ($i=1,\dots,N$) in Assumption \ref{assumption-nonoverlap} together with some careful manipulations as presented in subsection \ref{apd-BCD-subpros}, one can show that all subproblems in our dual BCD admit closed-form solutions, leading to the following explicit iterative scheme:
\begin{equation}\label{algoBCDex1}
\begin{aligned}
\bm{y}^{(i),\ell+1}
& = \varepsilon \log \bm{b}^{(i)} - \varepsilon\log \Big( \cA^{(i)}\Big( \widetilde{M} \circ \exp \big( \varepsilon^{-1}\mbox{$\sum_{q=1}^{i-1} $} \cA^{(q,*)}\bm{y}^{(q),\ell+1} \\
&\hspace{0.8cm} + \varepsilon^{-1}\mbox{$\sum_{q=i+1}^{N} $} \cA^{(q,*)}\bm{y}^{(q),\ell}\big) \circ \exp\big (\varepsilon^{-1}W^\ell \big) \Big) \Big),\quad 1\leq i \leq N,\\
W^{\ell+1} & = \min\Big\{ \varepsilon\log\Big( U./\big(\widetilde{M}\circ \exp\big(\varepsilon^{-1}{\textstyle\sum_{q=1}^{N}}\cA^{(q,*)}\bm{y}^{(q),\ell+1}\big) \big) \Big),\,0   \Big\}.
\end{aligned}
\end{equation}
Alternatively, for any $\ell\geq0$, let $\bm{\xi}^{(i),\ell} := \exp\big( \varepsilon^{-1}\bm{y}^{(i),\ell} \big)$ for $i=1,\ldots,N$ and $\Gamma^\ell := \exp\big(  \varepsilon^{-1}W^\ell \big)$, then the iterative scheme \eqref{algoBCDex1} can be equivalently written as
\begin{equation}\label{algoBCDex2}
\begin{aligned}
\bm{\xi}^{(i),\ell+1} &= \bm{b}^{(i)}./ \cA^{(i)}\Big( \widetilde{M}\circ (\cA^{(1,\bullet)}\bm{\xi}^{(1),\ell+1})\circ\dots \circ (\cA^{(i-1,\bullet)}\bm{\xi}^{(i-1),\ell+1}) \\
&\hspace{1.2cm} \circ (\cA^{(i+1,\bullet)}\bm{\xi}^{(i+1),\ell}) \circ \dots \circ (\cA^{(N,\bullet)}\bm{\xi}^{(N),\ell})\circ \Gamma^\ell \Big),\quad 1\leq i\leq N,\\
\Gamma^{\ell+1} & = \min\Big\{ U./\big( \widetilde{M}\circ (\cA^{(1,\bullet)}\bm{\xi}^{(1),\ell+1})\circ\dots \circ (\cA^{(N,\bullet)}\bm{\xi}^{(N),\ell+1}) \big),\,1 \Big\}.
\end{aligned}
\end{equation}
Here, for any $\bm{z}\in\R^{m_i}$, the tensor $\cA^{(i,\bullet)}\bm{z}\in \R^{n_1\times n_2\times n_3}$ is defined as follows:
\begin{equation*}
(\cA^{(i,\bullet)}\bm{z})_{rst}
= \left\{\begin{array}{ll}
(\cA^{(i,*)}\bm{z} )_{rst}, &\mbox{if $(r,s,t)\in\cJ^{(i)}$}, \\[5pt]
1, &\mathrm{otherwise},
\end{array}\right.
\end{equation*}
where $\cJ^{(i)}$ is the aggregated non-zero pattern of $\cA^{(i,*)}$ defined by
\begin{equation}\label{nnzpa}
\cJ^{(i)} = \big\{\,(r,s,t) \mid (A^{(i)}_j)_{rst}\neq0 ~~\mathrm{for~some}~~j \in\{ 1,\dots,m_i\}\,\big\}.
\end{equation}

\begin{remark}
For the efficient implementation of \eqref{algoBCDex2}, it is more convenient to introduce the following tensors for  $1\leq i\leq N$:
\begin{equation*}
\begin{aligned}
\widehat{M}^{(i),\ell+1} &=  \widetilde{M}\circ (\cA^{(1,\bullet)}\bm{\xi}^{(1),\ell+1})\circ\dots \circ (\cA^{(i-1,\bullet)}\bm{\xi}^{(i-1),\ell+1})\circ (\cA^{(i+1,\bullet)}\bm{\xi}^{(i+1),\ell}) \circ \dots \circ (\cA^{(N,\bullet)}\bm{\xi}^{(N),\ell})\circ \Gamma^\ell, \\
\widehat{M}^{(N+1),\ell+1} &=  \widetilde{M}\circ (\cA^{(1,\bullet)}\bm{\xi}^{(1),\ell+1})\circ\dots \circ (\cA^{(N,\bullet)}\bm{\xi}^{(N),\ell+1}).
\end{aligned}
\end{equation*}
Then, the $\ell$-th cycle of the BCD scheme in \eqref{algoBCDex2} can be carried out as follows.
\begin{equation*}
\begin{array}{ll}
\widehat{M}^{(1),\ell+1} = \Big(\widehat{M}^{(N+1),\ell} ./ (\cA^{(1,\bullet)}\bm{\xi}^{(1),\ell})\Big)\circ\Gamma^{\ell},
&\bm{\xi}^{(1),\ell+1} = \bm{b}^{(1)}./ \cA^{(1)} \big(\widehat{M}^{(1),\ell+1}\big),\\ [5pt]
\widehat{M}^{(i),\ell+1} = \Big(\widehat{M}^{(i-1),\ell+1} ./ (\cA^{(i,\bullet)}\bm{\xi}^{(i),\ell})\Big) \circ (\cA^{(i-1,\bullet)}\bm{\xi}^{(i-1),\ell+1}),
&\bm{\xi}^{(i),\ell+1} = \bm{b}^{(i)}./ \cA^{(i)} \big(\widehat{M}^{(i),\ell+1}\big),  \quad 2\leq i\leq N, \\ [5pt]
\widehat{M}^{(N+1),\ell+1} = \big(\widehat{M}^{(N),\ell+1} ./ \Gamma^{(\ell)} \big)\circ (\cA^{(N,\bullet)}\bm{\xi}^{(N),\ell+1}),
&\Gamma^{\ell+1} = \min\big\{ U./\widehat{M}^{(N+1),\ell+1}, \,1 \big\}.
\end{array}
\end{equation*}
Note that in the actual implementation of \eqref{algoBCDex2}, only a single  tensor is used to  store $\widehat{M}^{(i),\ell+1}$ for $i=1,\dots,N+1$, and it is repeatedly overwritten and updated.
\end{remark}

Note that both iterative schemes \eqref{algoBCDex1} and
\eqref{algoBCDex2} are simple and easy-to-implement. The main computational complexity
for \eqref{algoBCDex2} is $\mathcal{O}(n_1 n_2 n_3)$. In particular, since the iterative scheme \eqref{algoBCDex2} only needs elementwise multiplicatons and the simple $\min(\cdot)$ operation, it can be much more efficient than \eqref{algoBCDex1} in practice. However, like Sinkhorn's algorithm, \eqref{algoBCDex2} may also suffer from numerical instabilities when $\varepsilon$ takes a small value. Hence, in the unlikely event where $\varepsilon$ is a small value in our iEPPA, one can use \eqref{algoBCDex1} instead to carry on all computations in the log domain and perform the \textit{log-sum-exp} (see, e.g., \cite[Section 4.4]{pc2019computational}) technique for avoiding underflow/overflow. In general, the proximal parameter $\varepsilon_k$ in our iEPPA does not need to be very small to obtain an accurate solution of the original problem \eqref{cmotproblem3dre} (hence \eqref{cmotproblem}) in a fairly fast speed. This is also evident from our experiments which indicate that $\varepsilon = 0.05$ is sufficient for obtaining a good performance. Therefore, we can safely use the efficient iterative scheme \eqref{algoBCDex2} as a subroutine in our iEPPA.

In addition, we notice that there is a close connection between Dykstra's algorithm with Bregman projection (including DyKL as a special case) and (block) coordinate descent methods, although to the best of our knowledge, such a connection has not been stated explicitly until the recent work by Tibshirani \cite{t2017dykstra}. Indeed, one can deduce from \cite[Section 5]{t2017dykstra} that the DyKL used in \cite{bccnp2015iterative} for solving the entropic regularized problem in form of \eqref{subcmot3d} (see Appendix \ref{apdDyKL} for the DyKL applied to the 2-marginal capacity constrained OT problem) is equivalent to the BCD method applied to the following dual problem
\begin{equation}\label{dualkl}
\min\limits_{\Lambda_i\in\mathbb{R}^{n_1\times n_2\times n_3}, \,i=1,\dots,N+1}\,\Phi^*\left(\nabla \Phi(K)-{\textstyle\sum^{N+1}_{i=1}}\Lambda_i\right)
+ {\textstyle\sum^{N+1}_{i=1}}\delta_{\mathcal{S}_i}^*(\Lambda_i),
\end{equation}
where $\Phi(X):=\sum_{rst}X_{rst}\log X_{rst}$, $K:=S\circ\exp(-C/\varepsilon)$, $\mathcal{S}_i:=\{X\in\mathbb{R}^{n_1\times n_2\times n_3}:\cA^{(i)}(X) = b^{(i)}\}$ for $i=1,\dots,N$, $\mathcal{S}_{N+1}:=\{X\in\mathbb{R}^{n_1\times n_2\times n_3}: X \leq U\}$, and $\delta_{\mathcal{S}_i}^*$  is the conjugate of the indicator function $\delta_{\mathcal{S}_i}$.
It is clear that the dual problem \eqref{dualkl} is different from ours in \eqref{subcmot3dredual}. Therefore, the DyKL and our dual BCD are not equivalent to each other. Moreover, our dual BCD (either \eqref{algoBCDex1} or \eqref{algoBCDex2}) consumes much less memory. Because our dual variables $(\bm{y}^{(1)},\dots,\bm{y}^{(N)},W)$ only need $\sum^N_{i=1}m_i+n_1n_2n_3$ units of memory, while the dual variables $(\Lambda_1,\dots,\Lambda_{N+1})$ in \eqref{dualkl} need $(N+1)n_1n_2n_3$ units of memory.

\subsection{Convergence results for dual BCD}

We next show the convergence results for our dual BCD in Algorithm \ref{algBCD}. It is worth noting that the (block) coordinate descent method enjoys a long history for solving the problem (containing \eqref{subcmot3dredual} as a special case) of minimizing a class of convex differentiable functions over a certain closed convex set; see, for example, \cite{lt1992convergence,lt1993convergence,tseng1993dual}. Hence, our main convergence results are simply derived by revisiting these classic works. To this end, we first show the existence of optimal solutions of problems \eqref{subcmot3d} and \eqref{subcmot3dredual}, and their relations in the following proposition whose proof can be found in Appendix \ref{apd-BCD-proof1}.

\begin{proposition}\label{existopt}
Suppose that Assumption \ref{assumpfeas} holds. Then, the optimal solutions of problems \eqref{subcmot3d} and \eqref{subcmot3dredual} exist. Moreover, for any optimal solution $\big(\bar{\bm{y}}^{(1)},\dots,\bar{\bm{y}}^{(N)},\widebar{W}\big{)}$ of problem \eqref{subcmot3dredual},
\begin{equation}\label{recprisol}
\widebar{X} := \exp\Big(\varepsilon^{-1} \Big(\, {\textstyle\sum_{i = 1}^N} \cA^{(i,*)}\bar{\bm{y}}^{(i)}+\widebar{W}-M\,\Big)\Big)
\end{equation}
is the optimal solution of problem \eqref{subcmot3d}.
\end{proposition}

Next we present the main convergence results for our dual BCD based on the
theory developed in \cite{lt1992convergence}. To make the paper self-contained, we provide its proof in Appendix \ref{apd-BCD-proof2}.

\begin{theorem}[\textbf{Convergence of dual BCD}]\label{thmBCD}
Let $\big{\{}\big{(}\bm{y}^{(1),\ell},\dots,\bm{y}^{(N),\ell},W^{\ell}\big{)}\big{\}}$ be the sequence generated by the dual BCD method in Algorithm \ref{algBCD}, and let $X^{\ell} := \exp\Big(\varepsilon^{-1} \Big(\,{\textstyle\sum_{i=1}^N}\cA^{(i,*)}\bm{y}^{(i),\ell}
+W^{\ell}-M\,\Big)\Big)$.
Then, the following statements hold.
\begin{itemize}
\item[{\rm (i)}] $\big{\{}\big{(}\bm{y}^{(1),\ell},\dots,\bm{y}^{(N),\ell},W^{\ell}\big{)}\big{\}}$ converges R-linearly to an optimal solution of problem \eqref{subcmot3dredual}.

\item[{\rm (ii)}] $\{X^{\ell}\}$ converges R-linearly to an optimal solution of problem \eqref{subcmot3d}.
\end{itemize}
\end{theorem}

\subsection{Implementable verification of condition \eqref{EPPAcond1}}\label{sec-verif}

From the previous subsection, we know that the dual BCD can be efficiently applied for solving the subproblem \eqref{EPPAsubpro} in our iEPPA. In this subsection, we shall discuss how to verify condition \eqref{EPPAcond1} at a point returned by the dual BCD.

We first assume that there is a mapping $\mathcal{G}:\R^{n_1\times n_2\times n_3}\rightarrow \R^{n_1\times n_2\times n_3}$ such that for any $0\leq X\leq U$, $\mathcal{G}(X)\in\Omega$ and $\|\mathcal{G}(X)-X\|_F \leq c\sum^N_{i=1}\|\bm{r}^{(i)}\|$, where $c>0$ is a constant  depending only on $\mathcal{G}$ and $\bm{r}^{(i)}:=\bm{b}^{(i)} - \cA^{(i)}(X)$ $(i = 1,\dots,N)$ are residuals. Since $\Omega$ is a polyhedron, such a mapping is typically definable in practice. We give three examples as follows.

\begin{example}
If the projection of a point $X$ onto $\Omega$ (denoted by $P_{\Omega}(X)$) is easy to compute, one can directly use $\mathcal{G}=P_{\Omega}$. In this case, from the Hoffman error bound theorem \cite{h1952on}, there exists a constant $c>0$ such that $\|\mathcal{G}(X)-X\|_F \leq c\sum^N_{i=1}\|\bm{r}^{(i)}\|$.
\end{example}

\begin{example}
Suppose that a relative interior point $X^{\mathrm{ri}}$ of $\Omega$ is available on hand, i.e., $\mathcal{A}^{(i)}(X^{\rm ri})=\bm{b}^{(i)}$, $i=1,\dots,N$, and $0<X^{\rm ri}<U$. Note that such a point can be obtained for many choices of $\mathcal{A}^{(i)}$ and $U$. For example, in the 2-marginal COT problem, $\Omega$ is formed by $\{X\in\mathbb{R}^{m \times n}: X\bm{1}_n=\bm{a}, ~X^{\top}\bm{1}_m=\bm{b}, ~0\leq X\leq U\}$. If $U>\bm{a}\bm{b}^{\top}$, then $\bm{a}\bm{b}^{\top}$ is obviously a relative interior point. Otherwise, one can apply the alternating projection method or its variants to find a point in the intersection of $\{X\in\mathbb{R}^{m \times n}: X\bm{1}_n=\bm{a}\}$, $\{X\in\mathbb{R}^{m \times n}: X^{\top}\bm{1}_m=\bm{b}\}$ and $\{X\in\mathbb{R}^{m \times n}: \epsilon\leq X\leq U-\epsilon\}$ with some small $\epsilon>0$. Having an available relative interior point $X^{\mathrm{ri}}$ on hand, we can then perform the following procedure. We first compute the projection of $X$ onto $\widebar{\Omega}:=\{X\in\mathbb{R}^{m \times n} :\mathcal{A}^{(i)}(X)=\bm{b}^{(i)}, ~i=1,\dots,N\}$, which is in general easier than $P_{\Omega}$. Let $Z:=P_{\widebar{\Omega}}(X)$ and $V':=Z-X$. It follows from the Hoffman error bound theorem that $\|V'\|_F\leq c'\sum^N_{i=1}\|\bm{r}^{(i)}\|$ with some $c'>0$. Then, if $0 \leq Z \leq U$, we are done. Otherwise, \blue{we employ a pullback strategy to obtain a point $\widetilde{X}=Z+\lambda\,(X^{\rm ri}-Z)$ with some $\lambda \in[0,1]$}. It is easy to see that $\mathcal{A}^{(i)}(\widetilde{X})=\bm{b}^{(i)}$ for $i=1,\dots,N$. By choosing
\begin{equation*}
\lambda = \max\left\{ \max\limits_{(r,s,t)\in\mathcal{J}_1}\left\{\frac{Z_{rst}-U_{rst}}{Z_{rst}-X^{\rm ri}_{rst}}\right\},\,\max\limits_{(r,s,t)\in\mathcal{J}_2}\left\{\frac{-Z_{rst}}{X^{\rm ri}_{rst}-Z_{rst}}\right\}\right\},
\end{equation*}
where $\mathcal{J}_1:=\{(r,s,t):Z_{rst}>U_{rst}\}$ and $\mathcal{J}_2:=\{(r,s,t):Z_{rst}<0\}$, we can also ensure that $0\leq \widetilde{X}\leq U$ and hence $\widetilde{X}\in\Omega$. Moreover, note from $Z=X+V'$ and $0\leq X\leq U$ that
\begin{equation*}
\begin{aligned}
\lambda &= \max\left\{ \max\limits_{(r,s,t)\in\mathcal{J}_1}\left\{\frac{X_{rst}+V'_{rst}-U_{rst}}{Z_{rst}-X^{\rm ri}_{rst}}\right\},\,\max\limits_{(r,s,t)\in\mathcal{J}_2}\left\{\frac{-X_{rst}-V'_{rst}}{X^{\rm ri}_{rst}-Z_{rst}}\right\}\right\}  \\
&\leq \max\left\{ \max\limits_{(r,s,t)\in\mathcal{J}_1}\left\{\frac{V'_{rst}}{U_{rst}-X^{\rm ri}_{rst}}\right\},\,\max\limits_{(r,s,t)\in\mathcal{J}_2}\left\{\frac{-V'_{rst}}{X^{\rm ri}_{rst}}\right\}\right\}
\leq c''\|V'\|_F,
\end{aligned}
\end{equation*}
where $c''>0$ is a constant  depending only on $U$ and $X^{\rm ri}$. Then, we have
\begin{equation*}
\begin{aligned}
&\|\widetilde{X}-X\|_F  \;=\; \|Z+\lambda\,(X^{\rm ri}-Z)-X\|_F
\leq \lambda\,\|X^{\rm ri}-Z\|_F + \|Z-X\|_F \\
&\leq \, \lambda (\|X^{\rm ri}-X\|_F + \|Z-X\|_F) +  \|Z-X\|_F
\;\leq\; \lambda \|U\|_F + (1+\lambda)\|V'\|_F
\\
&\leq \,c''\|U\|_F\|V'\|_F + 2 \|V'\|_F \leq (c''\|U\|_F+2)\,c'\,{\textstyle\sum^N_{i=1}}\|\bm{r}^{(i)}\|.
\end{aligned}
\end{equation*}
Therefore, the above procedure can be used as $\mathcal{G}$, i.e., $\mathcal{G}(X) = \widetilde{X}$.
\end{example}

\blue{
\begin{example}
For the 3-marginal CMOT problem \eqref{cmotproblem3d}, we consider two cases.
\begin{itemize}[leftmargin=0.5cm]
\item When no upper bound $U$ is imposed \textit{or} $U$ is a trivial upper bound (e.g., $U$ is a matrix of all ones), a highly efficient rounding procedure \cite[Algorithm 2]{lin2022on} (an extension of
    \cite[Algorithm 2]{awr2017near} to the multi-marginal case) can be readily used as $\mathcal{G}$, whose main computational complexity is $\mathcal{O}(n_1 n_2 n_3)$.

\item When the upper bound $U$ is nontrivial, one can perform as follows. Similar to Example 2, let $X^{\mathrm{ri}}$ be a relative interior point of $\Omega$. We first apply the rounding procedure
    \cite[Algorithm 2]{lin2022on} on $X$ to obtain a point $Z$ in the set $\big\{X\in\mathbb{R}^{n_1 \times n_2 \times n_3}: \sum_{s,t}X_{rst} = a_r, ~r = 1, \dots, n_1, ~\sum_{r,t}X_{rst} = b_s, ~s = 1, \dots, n_2, ~\sum_{r,s}X_{rst} = c_t, ~t = 1, \dots, n_3, ~X\geq0\big\}$. If $Z \leq U$, we are done; otherwise, we employ a pullback strategy as Example 2 to obtain a point $\widetilde{X}=Z+\lambda\,(X^{\rm ri}-Z)$ with some $\lambda \in[0,1]$. Thus, such a procedure can be used as $\mathcal{G}$.
\end{itemize}
\end{example}
}

Now, suppose that we have a dual point $\big{(}\bm{y}^{(1),\ell+1},\dots,\bm{y}^{(N),\ell+1},W^{\ell+1}\big{)}$ given by our dual BCD at the $\ell$-th iteration, and computed a primal point by
\begin{equation*}
X^{\ell+1} := \exp\Big(\varepsilon^{-1} \Big(\,{\textstyle\sum_{i=1}^N}\cA^{(i,*)}\bm{y}^{(i),\ell+1}
+W^{\ell+1}-M\,\Big)\Big) \in\mathbb{R}^{n_1\times n_2 \times n_3}_{++}.
\end{equation*}
From the optimality condition of $W$-subproblem in Algorithm \ref{algBCD} and $\nabla\phi(X)=\log X$, one can verify that
\begin{numcases}{}
0 = C + \varepsilon (\log X^{\ell+1}-\log S) - W^{\ell+1} - {\textstyle\sum_{i=1}^N} \cA^{(i,*)}\bm{y}^{(i),\ell+1}
, \label{bcdcheck1} \\[1pt]
W^{\ell+1} \leq 0, ~~U-X^{\ell+1}\geq0, ~~\langle W^{\ell+1},\,U-X^{\ell+1}\rangle=0, \label{bcdcheck2} \\[1pt]
\bm{r}^{(i),\ell+1} = \bm{b}^{(i)} - \cA^{(i)}(X^{\ell+1}),~~ 1\leq i \leq N, \label{bcdcheck3}
\end{numcases}
where $\bm{r}^{(1),\ell+1},\dots,\bm{r}^{(N),\ell+1}$ are the residuals at the $\ell$-th iteration. Clearly, when $\bm{r}^{(i),\ell+1}=0$ for $i=1,\dots,N$, $X^{\ell+1}$ is an exact optimal solution. However, $\bm{r}^{(1),\ell+1},\dots,\bm{r}^{(N),\ell+1}$ are generally nonzero vectors.

Next, we perform the procedure $\mathcal{G}$ on $X^{\ell+1}$ to obtain that $\widetilde{X}^{\ell+1}=\mathcal{G}(X^{\ell+1})\in\Omega\subseteq\Omega^{\circ}$ and $\|V^{\ell+1}\|_F \leq c\sum^N_{i=1}\|\bm{r}^{(i),\ell+1}\|$ with $V^{\ell+1}:=\widetilde{X}^{\ell+1}-X^{\ell+1}$. Moreover, for any $Y\in\Omega^{\circ}$, we have
\begin{equation}\label{checksubdiff}
\begin{aligned}
&\quad \;\langle - W^{\ell+1} - {\textstyle\sum_{i=1}^N} \cA^{(i,*)}\bm{y}^{(i),\ell+1}, \,Y - \widetilde{X}^{\ell+1} \rangle = -\langle W^{\ell+1}, \,Y - \widetilde{X}^{\ell+1} \rangle \\
&= -\langle W^{\ell+1}, \,Y - U \rangle - \langle W^{\ell+1}, \,U - \widetilde{X}^{\ell+1} \rangle \leq - \langle W^{\ell+1}, \,U - X^{\ell+1} - V^{\ell+1} \rangle \\
&= \langle W^{\ell+1}, \,V^{\ell+1} \rangle \leq \|W^{\ell+1}\|_F\|V^{\ell+1}\|_F
\leq c\,\|W^{\ell+1}\|_F{\textstyle\sum^N_{i=1}}\|\bm{r}^{(i),\ell+1}\| \leq c\tilde{c}\,{\textstyle\sum^N_{i=1}}\|\bm{r}^{(i),\ell+1}\|,
\end{aligned}
\end{equation}
where the first equality follows because $Y,\,\widetilde{X}^{\ell+1}\in\Omega^{\circ}$ and hence $\langle{\textstyle\sum_{i=1}^N} \cA^{(i,*)}\bm{y}^{(i),\ell+1}, \,Y - \widetilde{X}^{\ell+1} \rangle=0$, the first inequality follows from $W^{\ell+1}\leq0$ and $Y-U\leq0$, the third equality follows from \eqref{bcdcheck2} and the last inequality follows because $\{W^{\ell}\}$ is convergent (by Theorem \ref{thmBCD}(i)) and hence must be bounded from the above by some constant $\tilde{c}>0$. Thus, \blue{letting $\nu_{\ell}:=c\tilde{c}\,{\textstyle\sum^N_{i=1}}\|\bm{r}^{(i),\ell+1}\|$}, we can obtain from \eqref{checksubdiff} that $- W^{\ell+1} - {\textstyle\sum_{i=1}^N} \cA^{(i,*)}\bm{y}^{(i),\ell+1}\in\blue{\partial_{\nu_{\ell}}}\delta_{\Omega^{\circ}}(\widetilde{X}^{\ell+1})$. This together with \eqref{bcdcheck1} implies that
\begin{equation*}
0 \in \blue{\partial_{\nu_{\ell}}}\delta_{\Omega^{\circ}}(\widetilde{X}^{\ell+1}) + C + \varepsilon\,(\nabla\phi(X^{\ell+1})-\nabla\phi(S)).
\end{equation*}
\blue{From this relation, we see that condition \eqref{EPPAcond1} is verifiable at the candidate $(X^{\ell+1}, \,\widetilde{X}^{\ell+1})$ and no error occurs on the left-hand-side in this case, i.e., $\Delta^\ell=0$. Then, our inexact condition \eqref{EPPAcond1} can be satisfied when both the primal feasibility accuracy ${\textstyle\sum^N_{i=1}}\|\bm{r}^{(i),\ell+1}\|$ and the Bregman distance $\mathcal{D}_{\phi}(\widetilde{X}^{\ell+1}, \,X^{\ell+1})$ are smaller than the specified tolerance parameters. In practical implementations, since the construction of $\widetilde{X}^{\ell+1}$ and the computation of the Bregman distance will incur additional overhead, one will not compute $\widetilde{X}^{\ell+1}$ at the \textit{early} stage of the dual BCD iteration since it is not needed in the algorithm. Specifically, one may start to compute $\widetilde{X}^{\ell+1}$ for checking the Bregman distance $\mathcal{D}_{\phi}(\widetilde{X}^{\ell+1}, \,X^{\ell+1})$ \textit{only} when the primal feasibility accuracy has decreased to a sufficiently small level. In this way, the overhead incurred will be reduced.}

In contrast, for {\it either} Teboulle's inexact condition \eqref{EPPAcond1-Te} \textit{or} Eckstein's inexact condition \eqref{EPPAcond1-Ec}, even though a feasible point $\widetilde{X}^{\ell+1}\in\Omega$ can be constructed successfully, one still \textit{cannot} verify condition \eqref{EPPAcond1-Te} or \eqref{EPPAcond1-Ec} at $\widetilde{X}^{\ell+1}$ because $\widetilde{X}^{\ell+1}$ may not lie in $\mathbb{R}^{n_1\times n_2 \times n_3}_{++}$ and hence $\nabla\phi$ may not be well defined at $\widetilde{X}^{\ell+1}$. \blue{Thus, such existing inexact conditions may not be easy to verify, even if one is willing to do the expensive computation. In this regard, our inexact condition \eqref{EPPAcond1} is more advantageous.}


\subsection{Computation of solutions of subproblems in dual BCD}\label{apd-BCD-subpros}

\blue{In this subsection, we provide more details on how to solve the subproblems efficiently in our dual BCD method via the special structures imposed on $\cA^{(i)}$ ($i=1,\dots,N$) in Assumption \ref{assumption-nonoverlap}.} Recall the iterative scheme in Algorithm \ref{algBCD}, for any $1\leq i\leq N$, $\bm{y}^{(i),\ell+1}$ is computed by solving an unconstrained minimization problem:
\begin{equation}\label{eq-yhat}
\begin{aligned}
\min_{\bm{y}^{(i)}} \; \left\{
\begin{aligned}
&\varepsilon \Inprod{\widetilde{M}}{ \exp \Big( \varepsilon^{-1}
\big( \cA^{(i,*)} \bm{y}^{(i)} + {\textstyle\sum_{q=1}^{i-1}}\cA^{(q,*)} \bm{y}^{(q),\ell+1} +
{\textstyle\sum_{q=i+1}^{N}} \cA^{(q,*)} \bm{y}^{(q),\ell}+ W^\ell \big)\Big)} \\
&-\inprod{\bm{y}^{(i)}}{\bm{b}^{(i)}}
\end{aligned}
\right\}.
\end{aligned}
\end{equation}
\blue{To solve this problem, we give the following auxiliary proposition.}

\begin{proposition}\label{prop-key}
Suppose that Assumption \ref{assumption-nonoverlap} holds. Then, for any tensor $M\in \R^{n_1\times n_2\times n_3}$ and any vector $\bm{y}^{(i)}\in \R^{m_i}$, we have
\begin{equation*}
\big\langle M, \,\exp\big( \varepsilon^{-1}\cA^{(i,*)}\bm{y}^{(i)} \big)\big\rangle
= \big\langle \cA^{(i)}(M), \,\exp\big( \varepsilon^{-1}\bm{y}^{(i)} \big) \big\rangle
~+ {\textstyle\sum_{(r,s,t)\not\in \cJ^{(i)}}} M_{rst},
\end{equation*}
where $\cJ^{(i)}$ is defined in \eqref{nnzpa}.
\end{proposition}
\begin{proof}
Note from Assumption \ref{assumption-nonoverlap} that $A_j^{(i)}$ only has binary entries (0 or 1) for all $j = 1,\dots,m_i$, and the non-zero patterns of $\big\{ A_j^{(i)} \mid j = 1,\dots,m_i \big\}$ do not overlap with each other. Thus, one can see that
\begin{equation*}
\begin{aligned}
\big(\exp\big( \varepsilon^{-1}\cA^{(i,*)}\bm{y}^{(i)} \big) \big)_{rst}
&=\big(\exp\big( \varepsilon^{-1} {\textstyle\sum_{j=1}^{m_i}}  A_j^{(i)}y_j^{(i)}  \big)\big)_{rst} \vspace{2mm}\\[5pt]
&=\left\{\begin{array}{ll}
{\textstyle\sum_{j=1}^{m_i}} \exp\big(  \varepsilon^{-1}y_j^{(i)} \big)  \big(A_j^{(i)}\big)_{rst}, &\mathrm{if}~(r,s,t)\in \cJ^{(i)}, \vspace{2mm}\\
\exp(0) = 1, & \mathrm{otherwise}.
\end{array}\right.
\end{aligned}
\end{equation*}
 Let $\kappa = \sum_{(r,s,t)\not\in \cJ^{(i)}} M_{rst}$. We then have that
\begin{equation*}
\begin{aligned}
&\big\langle M, \,\exp\big( \varepsilon^{-1}\cA^{(i,*)}\bm{y}^{(i)} \big) \big\rangle
~=~ \kappa \,+\, \sum_{(r,s,t)\in \cJ^{(i)}}  \sum_{j=1}^{m_i} \exp\big(  \varepsilon^{-1}y_j^{(i)} \big) \big(A_j^{(i)}\big)_{rst} M_{rst}
\\
&=~ \kappa \,+\, \sum_{j=1}^{m_i} \exp\big(\varepsilon^{-1}y_j^{(i)} \big)
\left(\sum_{(r,s,t)\in \cJ^{(i)}} \big(A_j^{(i)}\big)_{rst} M_{rst} \right)
~=~ \kappa \,+\, \sum_{j=1}^{m_i} \exp\big(\varepsilon^{-1}y_j^{(i)} \big)  \inprod{A^{(i)}_j}{M} \\
&=~ \kappa \,+\, \sum_{j=1}^{m_i} \exp\big(\varepsilon^{-1}y_j^{(i)} \big)  \big(\cA^{(i)} (M) \big)_j.
\end{aligned}
\end{equation*}
This completes the proof.
\end{proof}

Using Proposition \ref{prop-key}, we can reformulate the first term in the objective function of \eqref{eq-yhat} as below:
\begin{equation*}
\begin{aligned}
&\Inprod{\widetilde{M}}{ \exp \Big( \varepsilon^{-1}  \big( \cA^{(i,*)} \bm{y}^{(i)} + \mbox{$\sum_{q=1}^{i-1} $} \cA^{(q,*)} \bm{y}^{(q),\ell+1} + \mbox{$\sum_{q=i+1}^{N} $} \cA^{(q,*)} \bm{y}^{(q),\ell}+ W^\ell \big)\Big)}
\\[3pt]
=\, & \Inprod{\widetilde{M}\circ \exp \Big( \varepsilon^{-1} \big( \mbox{$\sum_{q=1}^{i-1} $} \cA^{(q,*)} \bm{y}^{(q),\ell+1} +
\mbox{$\sum_{q=i+1}^{N} $} \cA^{(q,*)} \bm{y}^{(q),\ell}+ W^\ell \big)\Big)}{ \exp \big( \varepsilon^{-1}  \cA^{(i,*)} \bm{y}^{(i)} \big)}
\\[3pt]
=\, & \big\langle \widetilde{M}^{(i),\ell}\circ \exp\big(\varepsilon^{-1} W^\ell \big), \,\exp \big( \varepsilon^{-1}  \cA^{(i,*)} \bm{y}^{(i)}\big)\big\rangle
=
\big\langle\cA^{(i)}\big( \widetilde{M}^{(i),\ell} \circ \exp\big (\varepsilon^{-1}W^\ell \big)\big), \,\exp\big(  \varepsilon^{-1}\bm{y}^{(i)} \big)\big\rangle + \Upsilon,
\end{aligned}
\end{equation*}
where $\widetilde{M}^{(i),\ell} := \widetilde{M}\circ\exp\big( \varepsilon^{-1}\sum_{q=1}^{i-1}\cA^{(q,*)}\bm{y}^{(q),\ell+1} + \varepsilon^{-1}\sum_{q=i+1}^{N}\cA^{(q,*)}\bm{y}^{(q),\ell}\big)$ and $\Upsilon$ is a constant independent of $\bm{y}^{(i)}$. Thus, $\bm{y}^{(i),\ell+1}$ can be simply computed by
\begin{equation*}
\begin{aligned}
\bm{y}^{(i),\ell+1}
&= \arg\min_{\bm{y}^{(i)}} \Big\{\varepsilon \big\langle\cA^{(i)}\big( \widetilde{M}^{(i),\ell} \circ \exp\big (\varepsilon^{-1}W^\ell \big) \big), \,\exp\big( \varepsilon^{-1}\bm{y}^{(i)} \big)\big\rangle
- \inprod{\bm{y}^{(i)}}{\bm{b}^{(i)}} \Big\}   \\
&= \varepsilon \log \bm{b}^{(i)} - \varepsilon\log \big( \cA^{(i)}\big( \widetilde{M}^{(i),\ell} \circ \exp\big (\varepsilon^{-1}W^\ell \big)\big)\big).
\end{aligned}
\end{equation*}
After obtaining $\bm{y}^{(i),\ell+1}$, $i = 1,\dots,N$, we then update $W^{\ell+1}$ by solving the following problem:
\begin{equation*}
\begin{aligned}
W^{\ell+1}
&=\arg\min_W\Big\{\varepsilon \big\langle\widetilde{M}, \,\exp\big( \varepsilon^{-1}\mbox{$\sum_{q=1}^{N} $}\cA^{(q,*)}\bm{y}^{(q),\ell+1} + \varepsilon^{-1}W \big)\big\rangle - \inprod{W}{U} + \delta_{-}(W) \Big\}, \\
&=\min\big\{ \varepsilon\log\big( U./(\widetilde{M}\circ Y^{\ell+1}) \big), \,0  \big\},
\end{aligned}
\end{equation*}
where $Y^{\ell+1} = \exp\big( \mbox{$\sum_{q=1}^{N}
$}\cA^{(q,*)}\bm{y}^{(q),\ell+1} \big)$.

\blue{
From the above discussions, one can see that the binary coefficient entries and the non-overlapping pattern imposed on $\cA^{(i)}$ ($i=1,\dots,N$) are vital for the efficient computation of solutions of the subproblems, and as we have mentioned after Assumption \ref{assumption-nonoverlap}, such special structures do appear in application problems such as the CMOT problem and the discrete tomography problem.

}

\section{Numerical experiments}\label{secnum}

In this section, we conduct numerical experiments to evaluate the performance of our iEPPA in Algorithm \ref{algEPPA}, which employs the dual BCD in Algorithm \ref{algBCD} as a subroutine, for solving the 2-marginal and 3-marginal CMOT problems \eqref{cmotproblem3d}. More details on applying the dual BCD for solving \eqref{subcmot3d} with the constraints in \eqref{cmotproblem3d} can be found in Appendix \ref{apdCMOT}. For the 2-marginal case, we compare our iEPPA with DyKL adapted in \cite{bccnp2015iterative} (see also in Appendix \ref{apdDyKL}) and the commercial solver Gurobi. For the 3-marginal case, we only compare our iEPPA with Gurobi. Moreover, we conduct experiments by applying our model \eqref{cmotproblem} for solving the discrete tomography problem \cite{w2006tomography}. All experiments are run in \blue{{\sc Matlab} R2021a} on a workstation with Intel Xeon processor E5-2680v3@2.50GHz (with 12 cores and 24 threads) and \blue{128GB of RAM, equipped with Linux OS}.

It is easy to show that the dual problem of \eqref{cmotproblem} is
\begin{equation}\label{eq-dcmot}
\max_{\bm{y}^{(1)},\dots,\bm{y}^{(N)},W}\;\left\{ {\textstyle\sum_{i=1}^N}\inprod{\bm{b}^{(i)}}{\bm{y}^{(i)}} +\inprod{U}{W} \;:\; {\textstyle\sum_{i=1}^N}\cA^{(i,*)}\bm{y}^{(i)}+ W\leq C,\; W\leq 0  \right\},
\end{equation}
and the Karush-Kuhn-Tucker (KKT) system for \eqref{cmotproblem} and \eqref{eq-dcmot} is
\begin{equation}\label{eq-kkts}
\begin{split}
& \cA^{(i)}(X) - \bm{b}^{(i)} = 0,\quad i = 1,\dots, N,\quad \mbox{$\sum_{i=1}^N$} \cA^{(i,*)}\bm{y}^{(i)}+ W\leq C,\quad 0\leq X\leq U,\\[5pt]
& \inprod{X}{\mbox{$\sum_{i=1}^N$} \cA^{(i,*)}\bm{y}^{(i)}+ W - C} = 0,\quad \inprod{W}{U-X} = 0,\quad W\leq 0.
\end{split}
\end{equation}
where $\bm{y}^{(1)},\dots,\bm{y}^{(N)}$ and $W$ are the Lagrangian multipliers (or dual variables). It is well known (for example, from \cite[Section 4.3]{bt1997introduction}) that when both the primal problem \eqref{cmotproblem} and the dual problem \eqref{eq-dcmot} are feasible, $(\widehat{X},\widehat{\bm{y}}^{(1)},\dots,\widehat{\bm{y}}^{(N)},\widehat{W})$ satisfies the KKT system \eqref{eq-kkts} if and only if $\widehat{X}$ solves the primal problem \eqref{cmotproblem} and $(\widehat{\bm{y}}^{(1)},\dots,\widehat{\bm{y}}^{(N)},\widehat{W})$ solves the dual problem \eqref{eq-dcmot}, respectively. Then, based on the KKT system \eqref{eq-kkts}, we define the relative KKT residual for any $\big(X,\bm{y}^{(1)},\dots,\bm{y}^{(N)},W \big)$ as follows:
\begin{equation*}
\Delta_{\rm kkt}\Big(:=\Delta_{\rm kkt}\big(X,\bm{y}^{(1)},\dots,\bm{y}^{(N)},W \big)\Big): = \max \big\{  \Delta_i \;:\; 1\leq i\leq 7 \big\},
\end{equation*}
where
\begin{align*}
& \Delta_1(X): = \frac{\big( \sum_{i = 1}^N\!\norm{\cA^{(i)}(X) - \bm{b}^{(i)}}^2 \big)^{1/2}}{1+\big( \sum_{i = 1}^N\norm{\bm{b}^{(i)}}^2 \big)^{1/2}},
~~\Delta_2\big(\bm{y}^{(1)},\dots,\bm{y}^{(N)},W \big): = \frac{\big\|\max\!\big\{ \!\sum_{i=1}^N\!\cA^{(i,*)}\bm{y}^{(i)}+ W - C,0 \big\}\big\|_F}{1+\norm{C}_F}, \\
& \Delta_3(X) := \frac{\norm{\min\{ X,0 \}}_F}{1+\norm{X}_F},
~~\Delta_4(X): = \frac{\norm{\min\{ U-X,0\}}_F}{1+\norm{U}_F},\quad  \Delta_5(W) := \frac{\norm{\max\{ W,0 \}}_F}{1+\norm{W}_F}, \\
& \Delta_6(X,W): = \frac{|\inprod{W}{U-X}|}{1+\norm{U}_F},
~~\Delta_7\big( X,\bm{y}^{(1)},\dots,\bm{y}^{(N)},W \big): = \frac{|\inprod{X}{\sum_{i=1}^N\cA^{(i,*)}\bm{y}^{(i)}+ W - C}|}{1+\norm{C}_F}.
\end{align*}
Obviously,  $ \big(X,\bm{y}^{(1)},\dots,\bm{y}^{(N)},W \big) $ is a solution of the KKT system \eqref{eq-kkts} if and only if $\Delta_{\rm kkt} = 0$.
We then use $\Delta_{\rm kkt}$ to set up the stopping criterion for the iEPPA. Specifically, we terminate the iEPPA when
\begin{equation*}
\Delta_{\rm kkt}\big(X^{k+1},\bm{y}^{(1),k+1},\dots,\bm{y}^{(N),k+1},W^{k+1} \big) < 10^{-5},
\end{equation*}
where $X^{k+1}$ and $\big(\bm{y}^{(1),k+1},\dots,\bm{y}^{(N),k+1},W^{k+1} \big)$ are the approximate optimal solutions of the subproblem \eqref{EPPAsubpro} and its corresponding dual problem, respectively, at the $k$-th iteration. The maximum number of iterations for the iEPPA is set to be 500.

The performance of our iEPPA naturally depends on the efficiency of the dual BCD for solving the subproblem \eqref{EPPAsubpro}. Therefore, the choice of the proximal parameter $\varepsilon$ and the stopping criterion for the subproblem at each iteration are vital for implementing the iEPPA. Note that a smaller regularization parameter $\varepsilon$ would lead to a more difficult subproblem, and moreover, a very small $\varepsilon$ may also cause numerical instabilities due to the loss of accuracy involving overflow/underflow operations.
In all our numerical experiments, we simply fix $\varepsilon=0.05$. With this choice, we would not encounter any numerical instability and can safely use the iterative scheme \eqref{algoBCDex2-1} to solve the subproblem efficiently.

\blue{
As discussed in subsection \ref{sec-verif}, our inexact condition \eqref{EPPAcond1} is verifiable and can be satisfied as long as both the primal feasibility accuracy ${\textstyle\sum^N_{i=1}}\|\bm{r}^{(i),k,\ell+1}\|$ ($\bm{r}^{(i),k,\ell+1}:=\bm{b}^{(i)} - \cA^{(i)}(X^{k,\ell+1})$, $i=1,\dots,N$) and the Bregman distance $\mathcal{D}_{\phi}(\widetilde{X}^{k,\ell+1}, \,X^{k,\ell+1})$ are sufficiently small, where $X^{k,\ell+1}$ is obtained by substituting $(\bm{y}^{(1),k,\ell+1},\dots,\bm{y}^{(N),k,\ell+1},W^{k,\ell+1})$ into \eqref{recprisol} at the $\ell$-th dual BCD iteration within the $k$-th outer iteration, and $\widetilde{X}^{k,\ell+1}$ can be constructed by $\widetilde{X}^{k,\ell+1}:=\mathcal{G}(X^{k,\ell+1})$ with a proper procedure $\mathcal{G}$ (see Example 3 in subsection \ref{sec-verif}). Note that, by such a construction, we have $\|\widetilde{X}^{k,\ell+1}-X^{k,\ell+1}\|_F \leq c\sum^N_{i=1}\|\bm{r}^{(i),k,\ell+1}\|$ for some constant $c>0$. Thus, when the primal feasibility accuracy $\sum^N_{i=1}\|\bm{r}^{(i),k,\ell+1}\|$ is small, the Bregman distance $\mathcal{D}_{\phi}(\widetilde{X}^{k,\ell+1}, \,X^{k,\ell+1})$ is also likely to be small, as always observed from our experiments. Since constructing $\widetilde{X}^{k,\ell+1}$ and calculating the Bregman distance $\mathcal{D}_{\phi}(\widetilde{X}^{k,\ell+1}, \,X^{k,\ell+1})$ explicitly is more costly than calculating the primal feasibility accuracy, such a phenomenon
then allows us to employ an economical way to check the condition $\mathcal{D}_{\phi}(\widetilde{X}^{k,\ell+1}, \,X^{k,\ell+1})\leq\mu_k$. Specifically, in our implementation, we \textit{first} compute the relative primal feasibility accuracy $\Delta_1(X^{k,\ell+1})$, and \textit{only start} to check $\mathcal{D}_{\phi}(\widetilde{X}^{k,\ell+1}, \,X^{k,\ell+1})\leq\mu_k$ when $\Delta_1(X^{k,\ell+1})\leq\widetilde{\mu}_k$ with $\{\widetilde{\mu}_k\}$ being a given summable positive sequence. This together with proper choices of $\{\widetilde{\mu}_k\}$ would help us to avoid the explicit construction of $\widetilde{X}^{k,\ell+1}$ and the computation of the Bregman distance as much as possible to save cost until the later stage of the dual BCD method, while enforcing our inexact condition \eqref{EPPAcond1} to guarantee the convergence of the iEPPA. In the following experiments, we set $\mu_k=\max\big\{(k+1)^{-1.1},\,10^{-6}\big\}$ and  $\widetilde{\mu}_k=\max\big\{10^{-4}\times\left(\frac{2}{3}\right)^k,\,10^{-6}\big\}$ for $k\geq0$. As we shall see later, such a simple checking strategy is enough to obtain a good practical performance.
}

It is well known that Gurobi is one of the most powerful and reliable solvers for solving LPs. Therefore, we use the solution obtained by Gurobi as a benchmark to evaluate the quality of solutions obtained by other methods. \blue{In our experiments, we use Gurobi (version 9.5.1 with an academic license) by only choosing the barrier method and disabling the presolving phase as well as the cross-over strategy so that Gurobi has the best performance. The reasons for choosing the aforementioned settings are three-fold. First, as observed from our experiments, other methods (such as the primal/dual simplex method) embedded in Gurobi are in general not as efficient as the barrier method. Second, we observe that when the presolving phase is enabled, Gurobi appears to be rather unstable and may fail to give a reasonably accurate solution for large-scale CMOT problems. Third, the cross-over strategy is usually too costly in our tests. }

\subsection{Experiments on synthetic data for 2-marginal CMOT}\label{subsecnum2d}

In this subsection, we consider the CMOT problem \eqref{cmotproblem3d} in the 2-marginal case and generate simulated examples to test each algorithm. For each example, we first generate two discrete probability distributions denoted by
\begin{equation*}
D_1 := \left\{ (a_r, \,\bm{p}_r)\in \R_+\times \R^3\;:\; r = 1,\dots,n_1 \right\} ~~ \mathrm{and} ~~
D_2 := \left\{ (b_s, \,\bm{q}_s)\in \R_+\times \R^3\;:\; s = 1,\dots,n_2
\right\}.
\end{equation*}
Here, $\bm{a}:=(a_1, \dots\!, a_{n_1})^{\top}$ and $\bm{b}:=(b_1, \dots\!, b_{n_2})^{\top}$ are probabilities/weights generated from the standard uniform distribution on the open interval $(0,\,1)$,  and further normalized such that $\sum^{n_1}_ra_r=\sum^{n_2}_sb_s=1$. Moreover, $\{\bm{p}_r\}$ and $\{\bm{q}_s\}$ are the support points whose entries are drawn from a Gaussian mixture distribution.
With these support points, the cost matrix $C$ is generated by $C_{rs} = \norm{\bm{p}_r-\bm{q}_s}^2$ for $1\leq r\leq n_1$ and $1\leq s\leq n_2$ and normalized by dividing (element-wise) by its maximal entry.

We next describe how to generate an upper bound matrix $U\in\mathbb{R}^{n_1 \times n_2}_{++}$. Note that if most of the entries of $U$ are too large (e.g., $U$ is a matrix of all ones), then such an upper bound matrix can be redundant. Conversely, if most of the entries of $U$ are too small, then the feasible set of \eqref{cmotproblem3d} can be empty and Assumption \ref{assumpfeas} fails to hold. Hence, a randomly generated upper bound matrix $U$ is usually unsatisfactory for our testing purpose. Thanks to the special structure of the
constraints in \eqref{cmotproblem3d}, one can easily see that $P:=\bm{a}\bm{b}^{\top}$ must lie in the set $\{X\in\mathbb{R}^{ n_1\times n_2}  : X\bm{1}_{n_2}=\bm{a}, ~X^{\top}\bm{1}_{n_1}=\bm{b}, ~X\geq0\}$. We then set $U:=2P=2\bm{a}\bm{b}^{\top}$ as the upper bound matrix. With this setting,  Assumption \ref{assumpfeas} can be satisfied and our numerical results also indicate that such an upper bound matrix is generally not redundant.

\subsubsection{Comparisons between Gurobi, iEPPA and DyKL}

In this part of experiments, we evaluate the performances of Gurobi, iEPPA and DyKL. For the DyKL, the entropic regularization parameter $\varepsilon$ is chosen from $\big\{10^{-1},10^{-2},10^{-3},10^{-4}\big\}$ in our numerical tests. For $\varepsilon\in\big{\{} 10^{-1},10^{-2}\big{\}}$, we follow
\cite[Section 5.2]{bccnp2015iterative} to implement the DyKL directly, while for $\varepsilon\in\big{\{} 10^{-3},10^{-4}\big{\}}$, we adapt the \textit{log-sum-exp} trick (see, for example,
\cite[Section 4.4]{pc2019computational}) to stabilize the DyKL (see Appendix \ref{apdDyKL} for the implementations of the DyKL). We terminate the DyKL when \blue{$\Delta_1(X^{k+1})<10^{-5}$}, where $X^{k+1}$ is generated by the DyKL at the $k$-th iteration.
Moreover, the maximum number of iterations for the DyKL is set to be 20000.

Table~\ref{tab:gaussianmixture} presents the computational results for different choices of $(n_1, n_2)$. In this table, ``normalized obj" denotes the normalized objective function value defined as \blue{$|\cF^k - \cF_{g}|/(1+|\cF_{g}|)$}, where $\cF_{g}$ denotes the objective value returned by Gurobi and $\cF^k$ is the approximate objective function value obtained by each algorithm; ``feasibility" denotes the primal feasibility accuracy, namely, \blue{$\max\{\Delta_1, \,\Delta_3, \,\Delta_4\}$}; ``time" denotes the total computational time (in seconds) and ``iter" denotes the number of iterations.
For our iEPPA, we also record the total number of dual BCD iterations. For instance, the item ``$14(418)$" means that iEPPA took 14 outer iterations with a total of 418 dual BCD iterations.

\begin{table}[h]
\centering
\captionsetup{width=15cm}
\caption{Numerical results on synthetic data for 2-marginal CMOT. In the table, ``g" stands for Gurobi; ``e" stands for iEPPA; ``d1",
``d2", ``d3", ``d4"  stand for DyKL with $\varepsilon = 10^{-1}$,  $10^{-2}$,  $10^{-3}$, $10^{-4}$, respectively.}\label{tab:gaussianmixture}
\renewcommand{\arraystretch}{1.05}
\tabcolsep 3.5pt
\blue{{\footnotesize
\begin{tabular}{|cc|cccccc|rrrrrr|}\hline
	$n_1$ & $n_2$ & g & e & d1 & d2 & d3 & d4 &
\multicolumn{1}{c}{g} & \multicolumn{1}{c}{e} & \multicolumn{1}{c}{d1} & \multicolumn{1}{c}{d2} & \multicolumn{1}{c}{d3} & \multicolumn{1}{c|}{d4} \\ \hline
	         &       & \multicolumn{6}{c|}{\texttt{normalized obj}} & \multicolumn{6}{c|}{\texttt{feasibility}}\\ \hline
4000 & 2000 & 0 & 5.7e-05 & 1.5e-02 & 3.7e-04 & 2.5e-04 & 2.5e-04 & 2.1e-13 & 9.9e-07 & 8.0e-06 & 9.8e-06 & 1.0e-05 & 1.0e-05 \\
4000 & 4000 & 0 & 6.2e-05 & 1.6e-02 & 3.8e-04 & 1.9e-03 & 1.9e-03 & 8.6e-16 & 9.9e-07 & 8.8e-06 & 9.9e-06 & 9.9e-06 & 1.0e-05 \\
4000 & 8000 & 0 & 7.2e-05 & 1.5e-02 & 5.3e-04 & 2.9e-04 & 3.0e-04 & 2.3e-14 & 1.0e-06 & 9.9e-06 & 9.8e-06 & 1.0e-05 & 1.0e-05 \\
5000 & 2500 & 0 & 5.4e-05 & 1.5e-02 & 2.6e-04 & 5.1e-04 & 5.1e-04 & 8.6e-14 & 9.8e-07 & 7.4e-06 & 9.8e-06 & 1.0e-05 & 1.0e-05 \\
5000 & 5000 & 0 & 6.3e-05 & 1.6e-02 & 3.9e-04 & 1.6e-04 & 1.7e-04 & 1.5e-13 & 9.9e-07 & 7.1e-06 & 9.6e-06 & 1.0e-05 & 1.0e-05 \\
5000 & 10000 & 0 & 5.4e-05 & 1.6e-02 & 3.5e-04 & 5.7e-04 & 5.8e-04 & 1.7e-15 & 9.8e-07 & 7.9e-06 & 9.9e-06 & 1.0e-05 & 1.0e-05 \\
6000 & 3000 & 0 & 5.2e-05 & 1.6e-02 & 3.2e-04 & 7.7e-04 & 7.7e-04 & 2.0e-14 & 9.9e-07 & 8.8e-06 & 9.8e-06 & 1.0e-05 & 1.0e-05 \\
6000 & 6000 & 0 & 5.5e-05 & 1.5e-02 & 3.6e-04 & 1.0e-03 & 1.0e-03 & 4.2e-14 & 9.9e-07 & 9.7e-06 & 1.0e-05 & 1.0e-05 & 1.0e-05 \\
6000 & 12000 & 0 & 5.1e-05 & 1.6e-02 & 3.2e-04 & 5.2e-04 & 5.4e-04 & 3.5e-14 & 9.8e-07 & 8.2e-06 & 9.9e-06 & 1.0e-05 & 1.0e-05 \\
7000 & 3500 & 0 & 6.2e-05 & 1.5e-02 & 4.4e-04 & 3.1e-04 & 3.2e-04 & 4.6e-14 & 9.8e-07 & 6.8e-06 & 9.7e-06 & 1.0e-05 & 1.0e-05 \\
7000 & 7000 & 0 & 6.7e-05 & 1.7e-02 & 4.3e-04 & 8.8e-04 & 9.1e-04 & 5.2e-15 & 9.8e-07 & 9.7e-06 & 9.9e-06 & 1.0e-05 & 1.0e-05 \\
7000 & 14000 & 0 & 4.1e-05 & 1.5e-02 & 2.2e-04 & 4.5e-04 & 4.6e-04 & 9.1e-14 & 9.9e-07 & 7.1e-06 & 1.0e-05 & 1.0e-05 & 1.0e-05 \\
		\hline
		     &       & \multicolumn{6}{c|}{\texttt{iter}} & \multicolumn{6}{c|}{\texttt{time (in seconds)}}\\ \hline
4000 & 2000 & - & 14(418) & 13 & 187 & 1120 & 11228 & 125.3 &  28.8 &   1.4 &  20.1 & 282.8 & 2714.2 \\
4000 & 4000 & - & 14(346) & 11 & 155 & 588 & 5911 & 169.4 &  61.4 &   3.4 &  48.9 & 361.1 & 3632.6 \\
4000 & 8000 & - & 15(281) & 10 & 135 & 857 & 8617 & 708.8 & 108.6 &   6.2 &  82.9 & 1038.1 & 10507.3 \\
5000 & 2500 & - & 12(300) & 13 & 162 & 951 & 9546 & 172.5 &  37.7 &   3.2 &  40.6 & 464.8 & 4614.3 \\
5000 & 5000 & - & 14(292) & 11 & 134 & 957 & 9549 & 527.3 &  74.2 &   5.3 &  65.1 & 912.3 & 9104.1 \\
5000 & 10000 & - & 14(258) & 10 & 128 & 737 & 7425 & 1198.4 & 137.1 &   9.0 & 112.7 & 1370.8 & 13865.0 \\
6000 & 3000 & - & 13(370) & 11 & 176 & 783 & 7873 & 249.9 &  60.7 &   3.8 &  58.8 & 538.8 & 5425.3 \\
6000 & 6000 & - & 14(412) & 13 & 233 & 891 & 9133 & 731.4 & 131.0 &   8.7 & 150.7 & 1213.0 & 12291.6 \\
6000 & 12000 & - & 14(243) & 10 & 129 & 780 & 7828 & 1686.5 & 190.5 &  12.9 & 157.8 & 2043.1 & 20847.3 \\
7000 & 3500 & - & 15(315) & 9 & 133 & 823 & 8316 & 352.5 &  77.1 &   4.3 &  61.5 & 771.7 & 7735.1 \\
7000 & 7000 & - & 14(229) & 8 & 107 & 543 & 5424 & 986.8 & 127.6 &   7.4 &  93.3 & 990.1 & 9971.4 \\
7000 & 14000 & - & 13(277) & 13 & 177 & 1109 & 11146 & 3239.9 & 266.2 &  22.6 & 292.3 & 3892.9 & 39098.4 \\
		\hline
\end{tabular}}}
\end{table}

From Table~\ref{tab:gaussianmixture}, one can observe that our iEPPA performs better than the DyKL in the sense that the iEPPA always returns a better approximate objective function value (using Gurobi as the benchmark) with a comparable feasibility accuracy in much less CPU time. The accuracy for the normalized objective function value returned by the iEPPA is always at the level of $10^{-5}$, while the accuracy of the DyKL is usually at the level of \blue{$10^{-4}$}. In particular, decreasing the value of $\varepsilon$ from $10^{-2}$ to $10^{-4}$ in the DyKL does not improve the accuracy for the objective function value \blue{significantly}, but is more time-consuming (this phenomenon is detailed more in Remark~\ref{rmk:accdykl}). Therefore, the DyKL \blue{and its stabilized variants may not be efficient for computing a relatively high precision solution of the original LP problem}.
Moreover, for large-scale problems, Gurobi is rather time-consuming \blue{and memory-consuming}. As an example, for the case where $(n_1, \,n_2)=(7000,\,14000)$ in Table \ref{tab:gaussianmixture}, a large-scale LP containing $9.8\times10^{7}$ box-constrained variables and 21000 equality constraints was solved. In this case, \blue{we observe that Gurobi is at least 10 times slower than our iEPPA, and it also needs about 55GB of RAM whereas our iEPPA only requires 15GB of RAM.}

\begin{remark}\label{rmk:accdykl}
For the DyKL, the accuracy of the solution in terms of the normalized objective function value is supposed to become better when the regularization parameter $\varepsilon$ becomes smaller. However, we only observe such a phenomenon when $\varepsilon$ is decreased from $10^{-1}$ to $10^{-2}$. When $\varepsilon \in \left\{ 10^{-2},10^{-3},10^{-4} \right\}$, the accuracy remains almost the same.
The reason is that the DyKL actually suffers from very slow convergence speed when $\varepsilon$ is small and hence the stopping tolerance $\text{Tol}_d=10^{-5}$ is not sufficient for the DyKL to obtain a good approximate solution. Indeed, when we set $\text{Tol}_d = 10^{-7}$ and test the DyKL with $\varepsilon = 10^{-3}$ on the case with $(n_1,\,n_2) = (4000, \,2000)$ (same as the first instance in Table~\ref{tab:gaussianmixture}), the returned normalized objective function value is \blue{$2.7\times 10^{-6}$}, which is much smaller than the accuracy \blue{($2.5\times10^{-4}$)} reported in Table \ref{tab:gaussianmixture}. However, the computational time also increases \blue{dramatically}.
\end{remark}

\subsubsection{Comparisons between Gurobi and iEPPA}\label{subsecEG}

To further evaluate the performance of our iEPPA, we conduct more experiments on synthetic data with support points generated by the same Gaussian mixture distribution as in the previous set of experiments. In the following experiments, \blue{we set $n_1=n_2=n$ and vary $n$ from 1000 to 9000.} The computational results are presented in Figure~\ref{fig-eppa_gurobi}. From the results, we see that the ``nobj" of iEPPA is always at the level of $10^{-5}$, which means that the objective function value returned by iEPPA is always close to that of Gurobi.
\blue{Moreover,} the computational time of iEPPA increases almost linearly with respect to $n$, \blue{while} the computational time taken by Gurobi grows much more rapidly than iEPPA. This is because when the problem size becomes large, the barrier method used in Gurobi may not be efficient enough and \blue{may also consume too much memory which may not be affordable on an ordinary PC.
In contrast, our iEPPA scales well in the sense that its computational time and memory consumption only grow at a low rate. Thus, it can be more favorable for solving the large-scale CMOT problem up to a moderate accuracy.}

\begin{figure}[h]
\begin{minipage}[!]{0.53\linewidth}
\centering
\hfill
\renewcommand{\arraystretch}{0.95}
\blue{
\begin{tabular}[!]{|c|c|cc|cc|}
\hline
\multicolumn{2}{|c|}{{\tt problem}} & \multicolumn{2}{c|}{\texttt{nobj}} & \multicolumn{2}{c|}{\texttt{feas}} \\
\hline
id& $n$   & g & e       & g       & e       \\ \hline
1 &  1000 & 0 & 6.2e-05 & 5.3e-14 & 1.0e-06 \\
2 &  2000 & 0 & 5.3e-05 & 2.1e-14 & 9.8e-07 \\
3 &  3000 & 0 & 5.6e-05 & 4.3e-14 & 9.8e-07 \\
4 &  4000 & 0 & 6.2e-05 & 8.6e-16 & 9.9e-07 \\
5 &  5000 & 0 & 6.3e-05 & 1.5e-13 & 9.9e-07 \\
6 &  6000 & 0 & 5.5e-05 & 4.2e-14 & 9.9e-07 \\
7 &  7000 & 0 & 6.7e-05 & 5.2e-15 & 9.8e-07 \\
8 &  8000 & 0 & 4.7e-05 & 1.7e-14 & 9.8e-07 \\
9 &  9000 & 0 & 5.7e-05 & 4.2e-14 & 1.0e-06 \\
\hline
\end{tabular}}
\end{minipage}
\quad
\begin{minipage}[!]{0.45\linewidth}
\centering
\includegraphics[width=7.5cm]{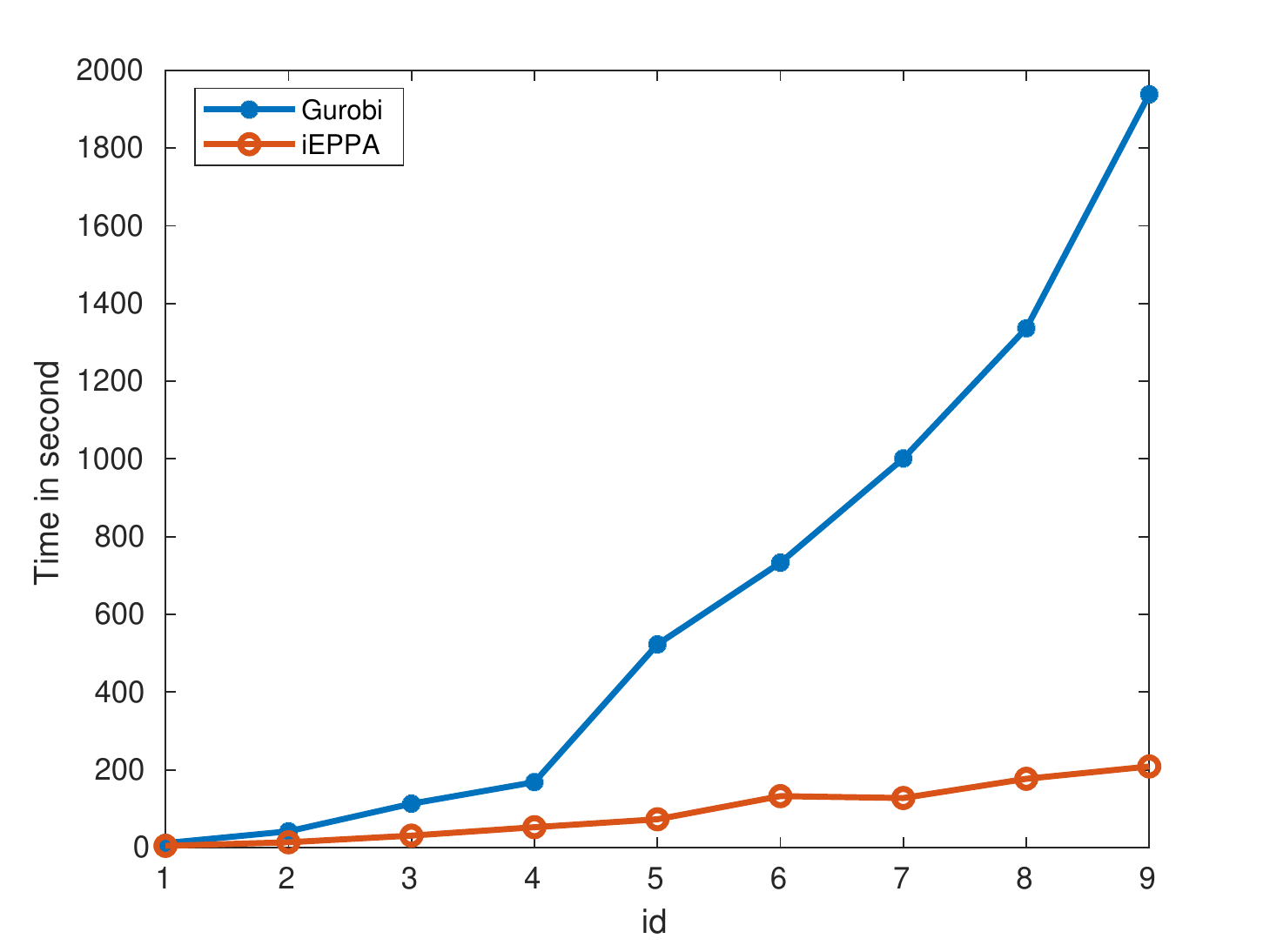}
\end{minipage}
\hfill
\captionsetup{width=15cm}
\caption{Comparisons between Gurobi and iEPPA for 2-marginal CMOT. In the table, ``nobj" denotes the normalized objective function value, ``feas" denotes the primal feasibility accuracy, ``g" stands for Gurobi and ``e" stands for iEPPA.}\label{fig-eppa_gurobi}
\end{figure}

\subsection{Experiments on synthetic data for 3-marginal CMOT}

In this subsection, we consider the standard CMOT problem \eqref{cmotproblem3d} in the 3-marginal case. Here, in view of the inferior performance of the DyKL presented in the last section, we only generate synthetic instances to evaluate the performance of our iEPPA against Gurobi to save space. Specially, we randomly generate three discrete probability distributions: $D_1=\left\{(a_r,\,\bm{p}_r)\in \R_+\times \R^3\;:\; r = 1,\dots, n_1 \right\}$, $D_2=\left\{(b_s,\,\bm{q}_s)\in \R_+\times \R^3\;:\; s = 1,\dots, n_2 \right\}$ and $D_3=\left\{(c_t,\,\bm{r}_t)\in \R_+\times \R^3\;:\; t = 1,\dots, n_3 \right\}$.
Similar to the 2-marginal case in subsection \ref{subsecnum2d}, the marginals $\bm{a}:=(a_1,\dots,a_{n_1})^{\top}$, $\bm{b}:=(b_1,\dots,b_{n_2})^{\top}$ and $\bm{c}:=(c_1,\dots,c_{n_3})^{\top}$ are all generated independently from a uniformly distribution on the interval $(0,1)$, respectively. Again, the marginals are normalized so that $\sum_{r=1}^{n_1}a_r=\sum_{s=1}^{n_2}b_s=\sum_{t=1}^{n_3}c_t=1$.  Moreover, the support points are generated independently from a Gaussian mixture distribution.
Given these support points, we  then compute the cost tensor $C$ as follows:
\begin{equation*}
C_{rst}: = \norm{\bm{p}_r-\bm{q}_s}^2+\norm{\bm{q}_s-\bm{r}_t}^2+\norm{\bm{r}_t-\bm{p}_r}^2,\quad \forall\;1\leq r\leq n_1,~1\leq s\leq n_2,~1\leq t\leq n_3.
\end{equation*}
We also normalize $C$ by dividing it by its maximal entry. To generate a reasonable upper bound $U$, we adapt the same strategy as in the 2-marginal case to set $U:=2\,(\bm{a} \otimes \bm{b} \otimes \bm{c})$. \blue{Figure~\ref{fig-eppa_gurobi_3d_1} presents comparisons between iEPPA and Gurobi, where we set $n_1 = n_2 = n_3 = n$ and vary $n$ from 50 to 500. Similar to the 2-marginal case in subsection \ref{subsecEG}, our iEPPA has better scalability with respect to the problem size for solving problems to a moderate accuracy. }


\begin{figure}[h]
\begin{minipage}[!]{0.5\linewidth}
\centering
\hfill
\blue{
\begin{tabular}[!]{|c|c|cc|cc|}
\hline
\multicolumn{2}{|c|}{{\tt problem}} & \multicolumn{2}{c|}{\texttt{nobj}} & \multicolumn{2}{c|}{\texttt{feas}} \\
\hline
id & $n$ & g & e  & g & e  \\ \hline
1 &    50 & 0 & 4.6e-05 & 1.1e-12 & 1.0e-06 \\
2 &   100 & 0 & 5.0e-05 & 7.0e-13 & 1.0e-06 \\
3 &   150 & 0 & 5.0e-05 & 1.1e-12 & 9.8e-07 \\
4 &   200 & 0 & 5.0e-05 & 1.3e-12 & 9.9e-07 \\
5 &   250 & 0 & 4.9e-05 & 4.5e-13 & 9.9e-07 \\
6 &   300 & 0 & 5.2e-05 & 3.2e-13 & 9.9e-07 \\
7 &   350 & 0 & 4.9e-05 & 1.2e-12 & 9.9e-07 \\
8 &   400 & 0 & 5.1e-05 & 1.1e-12 & 1.0e-06 \\
9 &   450 & 0 & 5.3e-05 & 1.4e-12 & 9.9e-07 \\
10&   500 & 0 & 5.7e-05 & 5.3e-13 & 9.7e-07 \\
\hline
\end{tabular}}
\end{minipage}
\quad
\begin{minipage}[!]{0.45\linewidth}
\centering
\includegraphics[width=7.5cm]{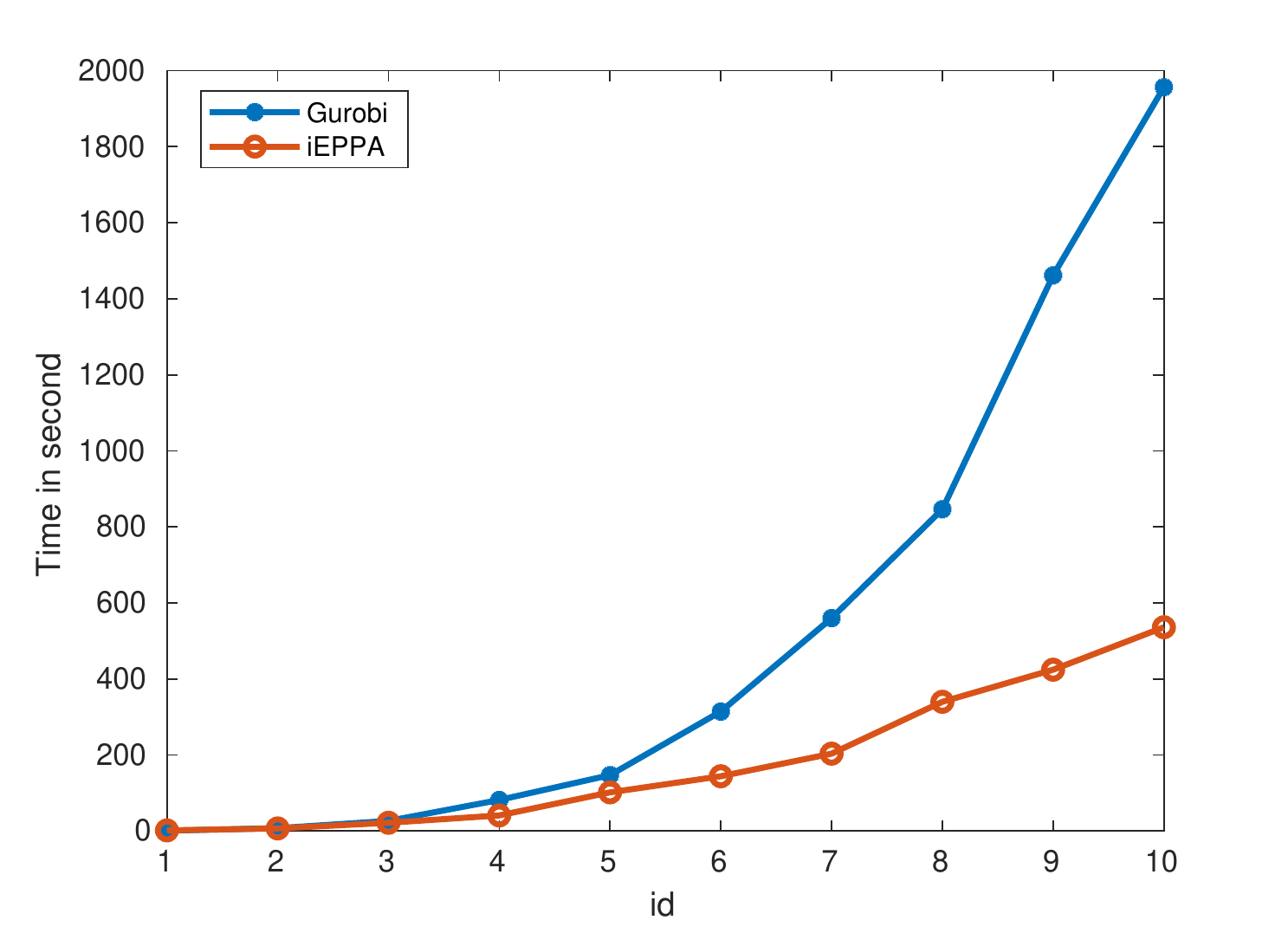}
\end{minipage}
\hfill
\captionsetup{width=15cm}
\caption{Comparisons between Gurobi and iEPPA for 3-marginal CMOT with {$n_1 = n_2=n_3=n$ and $ n\in \{50, 100, 150, 200, 250, 300, 350, 400, 450, 500\}$}.}\label{fig-eppa_gurobi_3d_1}
\end{figure}

\subsection{Experiments on an application to discrete tomography}\label{numsec-tomo}

In this subsection, we conduct experiments on discrete tomography to illustrate the modeling capability of our model \eqref{cmotproblem}.
We should mention that the purpose here is to present a preliminary investigation on the potential of using our model together with the iEPPA+BCD framework for solving the discrete tomography problem. A thorough numerical investigation is beyond the scope of this paper and will be left as a future research topic.

Let $X$ be a 2D image of size $n\times n$ and $\vec{v}$ be a given direction that takes the form $(1, p)$, $(1, -p)$, $(p, 1)$ or $(p, -1)$ with $p$ being a nonnegative integer. In our experiments, a tomographic projection on the image $X$ along $\vec{v}$ is constructed as follows: we view the image $X$ as a 2D grid of size $n \times n$, first pick all lines which are parallel to the direction $\vec{v}$ on this grid, then sum the entries on each line to form a vector. Such a projection then corresponds to a block of linear equality constraints of the form $\bm{b}^{(i)}=\mathcal{A}^{(i)}(X)$ in our model. Figure~\ref{cap1} shows the constructions of the tomographic projection along directions $(1,0)$ and $(2,1)$, respectively. More details on the construction can be found in Appendix \ref{appen3}.

\begin{figure}[h]
\centering
\includegraphics[width=0.3\textwidth]{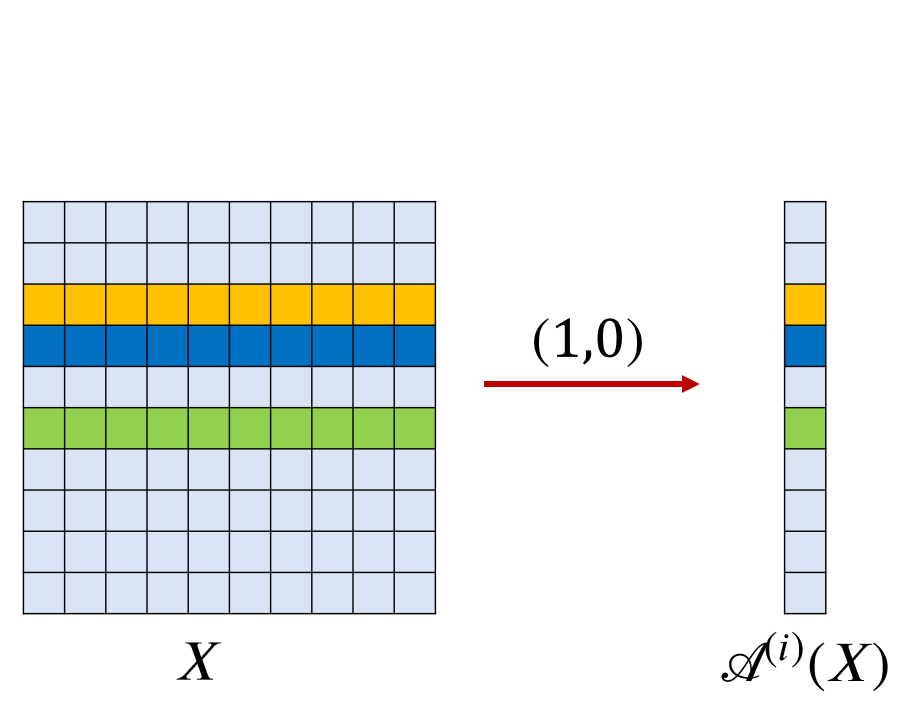}	
\hspace{2cm}	
\includegraphics[width=0.3\textwidth]{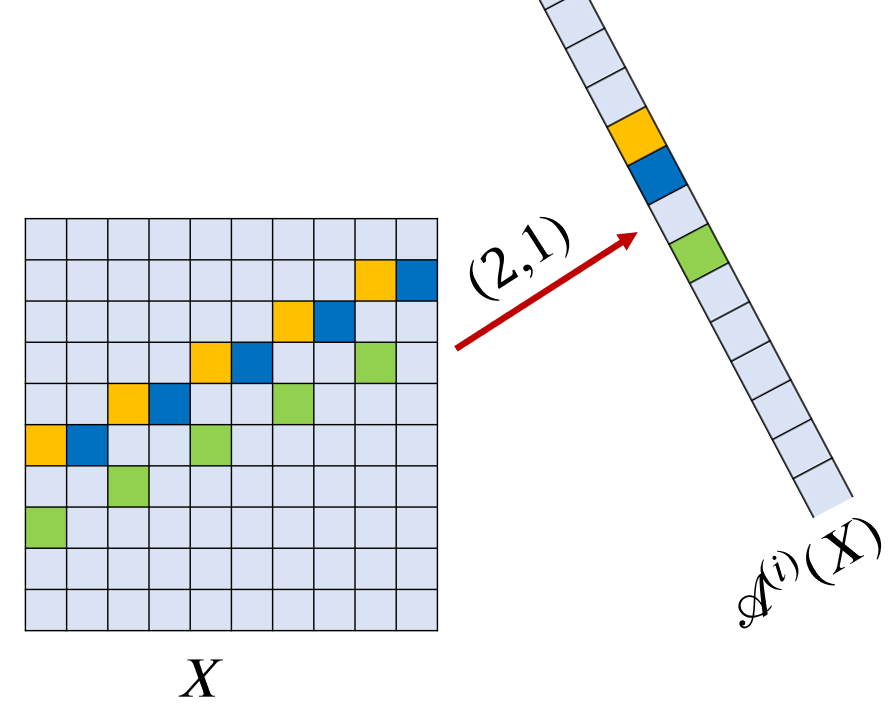}
\caption{Examples of the operator $\mathcal{A}^{(i)}$ for direction $\vec{v}=(1,0)$ (\textbf{left}) and direction $\vec{v}=(2,1)$ (\textbf{right}).}\label{cap1}
\end{figure}

In the following experiments, we will use five ground-truth images of size $256 \times 256$ (namely, $n=256$), as shown in the first row of Figure \ref{imag_data}. For each of them, we compute $N$ tomographic projections (which  correspond to the linear mappings $\cA^{(i)}$, $i=1,\dots,N$ in our model) on this image along different directions to obtain $\bm{b}^{(i)}$, $i=1,\dots,N$. Then, our goal is to recover the original image from the collection of projections $\{\bm{b}^{(i)}\}_{i=1}^N$ via applying our iEPPA+BCD framework for solving problem \eqref{cmotproblem}. Moreover, in our experiments, we set the entries of the cost matrix $C\in \R^{n\times n}$ to be $C_{rs} = |r-s|^2$ for $1\leq r,\,s\leq n$ and normalize $C$ by dividing (element-wise) it by its maximal entry.\footnote{The setting of the cost matrix $C$ may depend on the prior knowledge on the distribution of features in an image. Here, we simply set the entry of $C$ to be $C_{rs} = |r-s|^2$, $\forall\,1\leq r,\,s\leq n$, and normalize it for the preliminary testing purpose. More study on the choice of $C$ will be left in the future.} We do not use any upper bound matrix $U$ in the experiments.
Moreover, to quantify the reconstruction quality between the recovered solution \({X}^k\) and the ground-truth image \(X\), we evaluate the peak signal-to-noise ratio that is defined as $\mathtt{PSNR} = 10 \cdot \log_{10} \left(n^2 \cdot {\max\left\{X_{rs}\,:\, 1\leq r,s\leq n\right\}^2 }/{\norm{{X}^k-X}_F^2 } \right)$.

Figure~\ref{fig:tomo-time} presents the \(\mathtt{PSNR}\) values of the reconstructed images for $N\in\{10,20,\dots,90\}$. For a better visualization, we also show the reconstructed images corresponding to \(N\in\{20,50,80\}\) in Figure~\ref{imag_data}.
From the results, we observe that our model \eqref{cmotproblem} can faithfully recover the ground-truth image and the quality of the reconstructed image gradually improves when more  projections are used. Thus, to improve the quality of the reconstructed image, a straightforward way is to increase the number of projections. Fortunately, for our approach of using model \eqref{cmotproblem} and the iEPPA+BCD, imposing more projections would not increase the computational cost dramatically since the main computational unit (which is one BCD iteration) only depends on $N$ linearly. Specifically, for one more projection, we only need to add one more block of constraints in our model \eqref{cmotproblem} and then correspondingly add one more block of dual variables in the dual BCD method.

\begin{figure}[h]
\centering
{\includegraphics[width=0.6\textwidth]{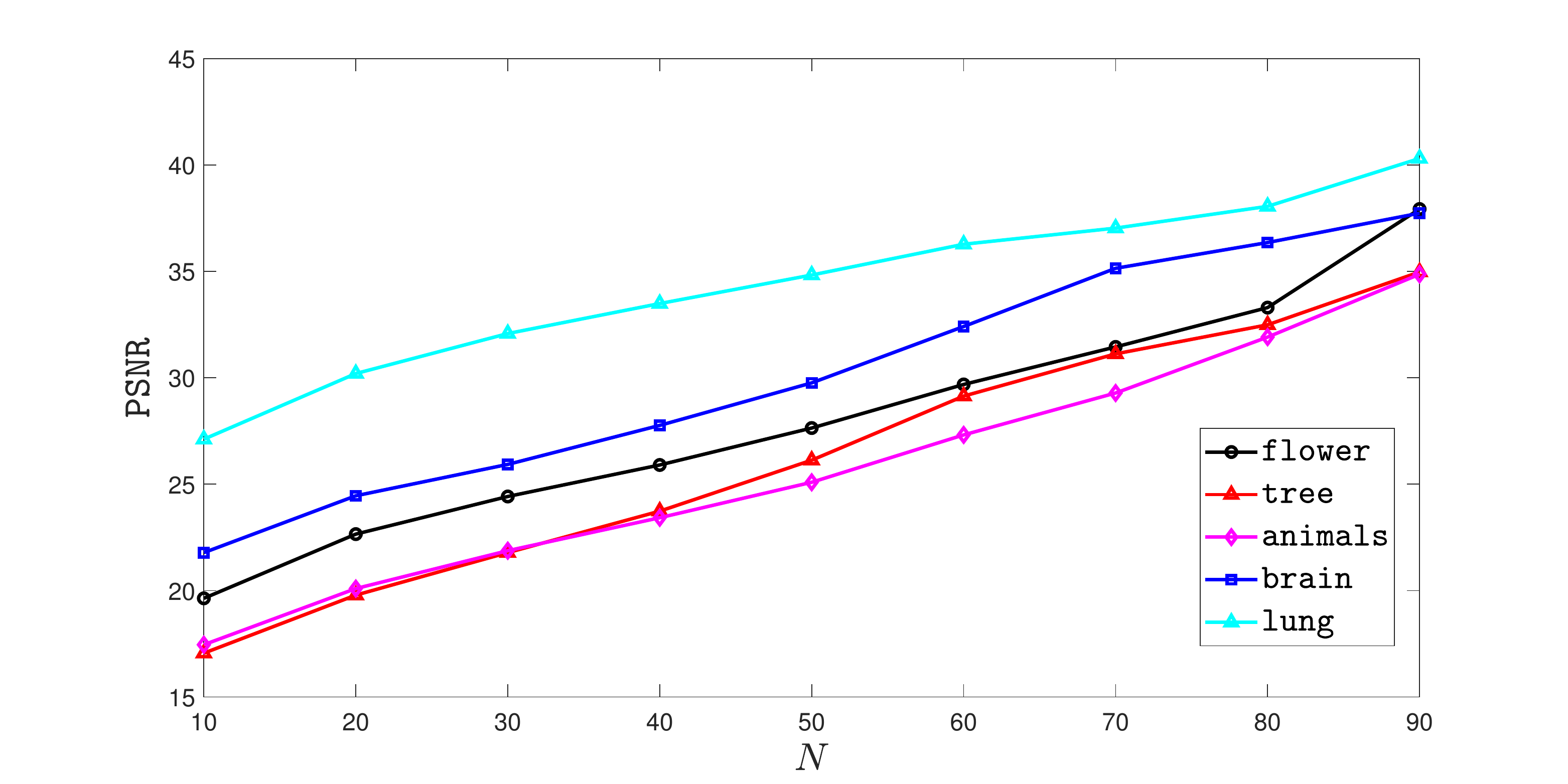}}
\caption{$\tt PSNR$ values of the reconstructed images for $N\in\{10,20,\dots,90\}$.}\label{fig:tomo-time}
\end{figure}

\begin{figure}[h]
\centering
\includegraphics[width=0.75\textwidth]{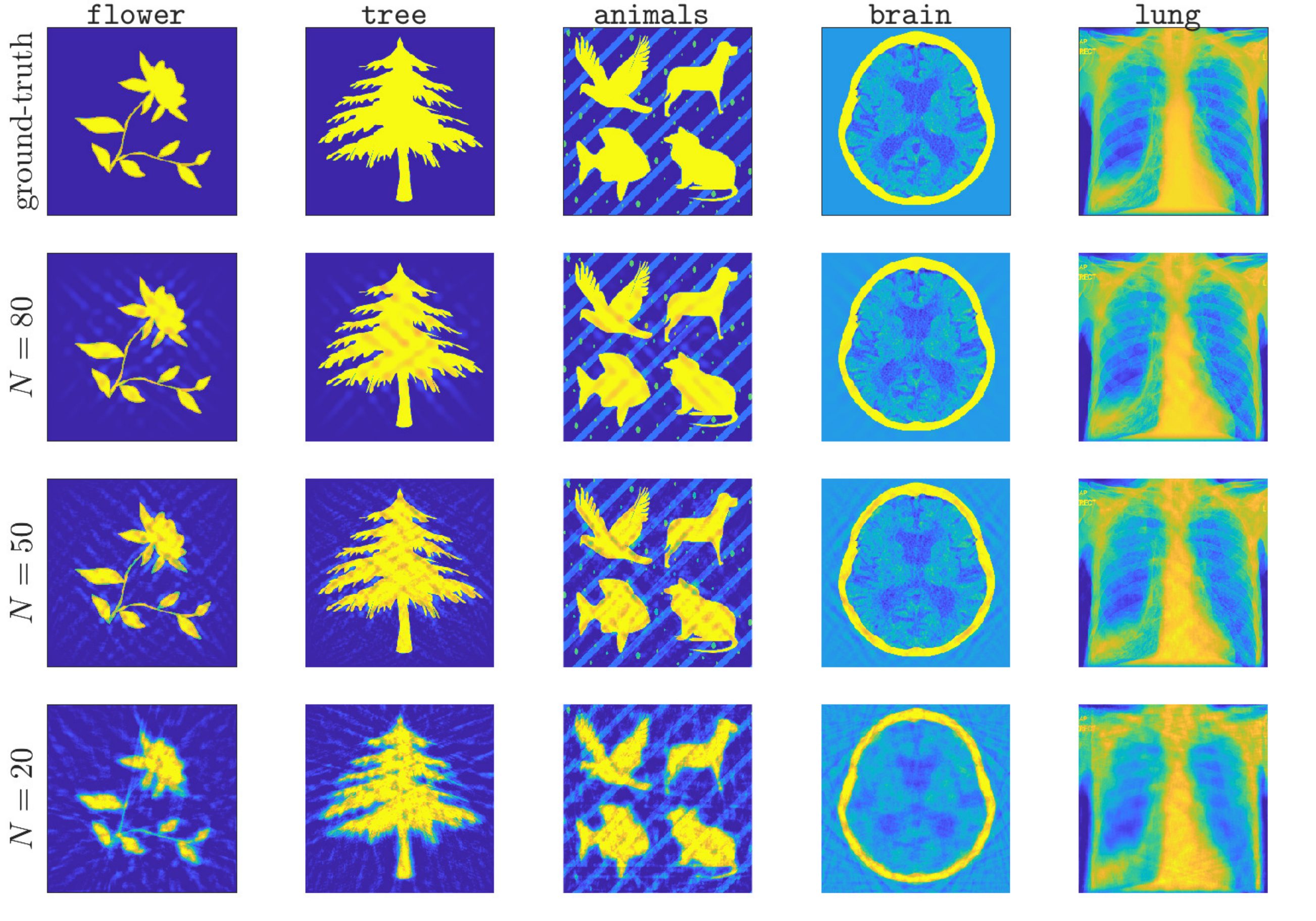}
\caption{\textbf{The first row}: ground-truth images of size $256 \times 256$. Here, \texttt{flower}, \texttt{tree} and \texttt{animals} are artificial images, while \texttt{brain} and \texttt{chest} are taken from \url{https://radiopaedia.org/images/9219097} and \url{https://radiopaedia.org/cases/loculated-pneumothorax}, respectively. \textbf{The second, third and fourth rows}: reconstructed images using 80, 50 and 20 projections.}\label{imag_data}
\end{figure}

\begin{remark}
To recover an $n \times n$ image by our model \eqref{cmotproblem} (which is an LP) with $N$ available projections, the corresponding (sparse) coefficient matrix of the equality constraint has at least the size of $Nn\,\times\,n^2$.
When $n$ and $N$ are large, such a large-scale problem can cause some LP solvers (e.g., Gurobi) to suffer from insufficient memory issues as well
as high computational cost on an ordinary PC. We note that another model (based on knowing a prior distribution) that aims to recover objects from a few tomographic projections is suggested in \cite{abraham2017tomographic,bergounioux2018variational}. In their framework, suppose that a 2D object with $N$ projections is available. Then, the decision variable for the corresponding multi-marginal optimal transport problem will be a tensor of the order $2+N$ in the formulation given in \cite{abraham2017tomographic}. Hence, it is difficult to implement the model efficiently. In addition, when the order of the tensor is large, the aforementioned model will invariably encounter memory issues. Moreover, our approach do not require any prior knowledge on the image to be recovered,
which is another key feature that makes our modeling framework even more attractive.
\end{remark}

\section{Concluding remarks}\label{seccon}

In this paper, we propose a class of linear programming (LP) problems that can be employed to efficiently model several application problems such as discrete tomography and disaggregation of input-output tables in economics. We then develop an implementable inexact entropic proximal point algorithm (iEPPA) for solving these specially structured LPs. To solve the subproblems that contain a special entropic proximal term, we adapt an easy-to-implement dual block coordinate descent (BCD) method to solve the associated more tractable dual subproblem. The convergence of our iEPPA and the R-linear convergence of the dual BCD method are also established. In particular, we develop a new practically verifiable inexact stopping condition for solving the iEPPA subproblem that has some computational advantages over those in the existing methods. Extensive numerical experiments have been conducted to demonstrate the high efficiency and robustness of our iEPPA+BCD framework for solving the capacity constrained multi-marginal optimal transport problem.
We also illustrate the potential modeling power of the proposed model by applying it to discrete tomography problems. Finally, we are aware of the classical works \cite{nielsen1993massively, nielsen1996solving} that applied the EPPA with specialized subsolvers for solving the two-stage and multi-stage stochastic network problems. It may be possible to extend our iEPPA+BCD framework for solving such special classes of LP problems. We will leave it as a future research topic.

\section*{Acknowledgments}

We thank the editor and referees for their valuable suggestions and comments, which have helped to improve the quality of this paper.

\begin{appendices}

\section{More details on the dual BCD}\label{apd-BCD}

\subsection{Proof of Proposition \ref{existopt}}\label{apd-BCD-proof1}
First, problem \eqref{subcmot3d} is equivalent to $\min_X\left\{\delta_{\Omega^{\circ}}(X)+
\langle {C}, \,X\rangle + \varepsilon\mathcal{D}_{\phi}(X, \,S)\right\}$.
Since $\mathrm{dom}\,\phi=\mathbb{R}_+^{n_1 \times n_2 \times n_3}$ and thus $\Omega^{\circ}\cap\mathrm{dom}\,\phi=\Omega$ is nonempty (by Assumption \ref{assumpfeas}) and bounded, then the objective function in the above problem is level bounded. Thus, a solution exists
\cite[Theorem 1.9]{rw1998variational} and must be unique since $\phi$ is strictly convex. The essential smoothness of $\phi$ further implies that the optimal solution can only lie in $\mathbb{R}_{++}^{n_1 \times n_2 \times n_3}$. Hence, the optimal solution of problem \eqref{subcmot3dre} also exists. Let $\big{(}\widebar{X}, \,\widebar{Z}\big{)}\in\mathbb{R}^{n_1 \times n_2 \times n_3}_{++}\times\mathbb{R}^{n_1 \times n_2 \times n_3}_{++}$ be an optimal solution of problem \eqref{subcmot3dre}. Since all constraint functions in \eqref{subcmot3dre} are affine and the set $\big{\{}\big{(}X,\,Z\big{)}\in\mathbb{R}^{n_1 \times n_2 \times n_3}\times\mathbb{R}^{n_1 \times n_2 \times n_3}: Z\geq0\big{\}}$ is a convex polyhedron, then it follows from \cite[Theorem 3.25]{r2006nonlinear} that there exist $\bar{\bm{y}}^{(i)}\in \R^{m_i},\;1\leq i\leq N$ and $\widebar{W}\in\mathbb{R}^{n_1 \times n_2 \times n_3}$ such that
\begin{numcases}{}
0 = M-\widebar{W}
- {\textstyle\sum_{i=1}^N} \cA^{(i,*)}\bar{\bm{y}}^{(i)}
+ \varepsilon \log \widebar{X}, \label{kktprimal1} \\
0 \in -\widebar{W} + \partial\delta_+(\widebar{Z}), \label{kktprimal2} \\
0 = \bm{b}^{(i)} - \cA^{(i)}(\widebar{X}),\quad 1\leq i \leq N, \label{kktprimal3}\\
0 = U - \widebar{X} - \widebar{Z}. \label{kktprimal4}
\end{numcases}
Note from \eqref{kktprimal1} that {$\widebar{X} = \exp\Big(\varepsilon^{-1} \,\Big({\textstyle\sum_{i=1}^N} \cA^{(i,*)}\bar{\bm{y}}^{(i)}+\widebar{W}-M\,\Big)\Big)$.} Then, substituting this and \eqref{kktprimal4} into \eqref{kktprimal2} and \eqref{kktprimal3}, recalling \eqref{subeqv} and the fact that $\partial\delta_{+}^* = \partial\delta_{-}$, one can see that
\begin{equation}\label{optcond}
\left\{\begin{aligned}
0 &= \bm{b}^{(i)} - \cA^{(i)}\Big( \exp\Big( \varepsilon^{-1} \big(\,{\textstyle\sum_{i=1}^N} \cA^{(i,*)}\bar{\bm{y}}^{(i)}
+\widebar{W}-M\,\big{)}\Big) \Big), \quad i=1,\dots,N, \\
0 &\in \exp\Big{(}\varepsilon^{-1} \big{(}\,{\textstyle\sum_{i=1}^N } \cA^{(i,*)}\bar{\bm{y}}^{(i)}
+\widebar{W}-M\,\big{)}\Big{)}
- U + \partial\delta_{-}(\widebar{W}).
\end{aligned}\right.
\end{equation}
This together with \cite[Theorem 3.5]{r2006nonlinear} implies that $\big{(}\bar{\bm{y}}^{(1)},\dots,\bar{\bm{y}}^{(N)},\widebar{W}\big{)}$ is an optimal solution of the dual problem \eqref{subcmot3dredual} and hence the optimal solution of problem \eqref{subcmot3dredual} exists.

Moreover, for any optimal solution $\big(\hat{\bm{y}}^{(1)},\dots,\hat{\bm{y}}^{(N)},\widehat{W}\big)$ of problem \eqref{subcmot3dredual}, it follows from \cite[Theorem 3.5]{r2006nonlinear} that it satisfies the system \eqref{optcond} in place of $\big{(}\bar{\bm{y}}^{(1)},\dots,\bar{\bm{y}}^{(N)},\widebar{W}\big{)}$. Let {$\widehat{X} := \exp\Big(\varepsilon^{-1} \Big(\, {\textstyle\sum_{i = 1}^N} \cA^{(i,*)}\hat{\bm{y}}^{(i)}+\widehat{W}-M\,\Big)\Big)$ and $\widehat{Z} := U - \widehat{X}$. By \eqref{subeqv} and the fact that $\partial\delta_{-}^* = \partial\delta_{+}$, it holds that  $\big(\widehat{X},\widehat{Z},\hat{\bm{y}}^{(1)},\dots,\hat{\bm{y}}^{(N)},\widehat{W} \big)$ satisfies the system \eqref{kktprimal1}--\eqref{kktprimal4}.} Thus it follows from \cite[Theorem 3.27]{r2006nonlinear} that $\big(\widehat{X},\,\widehat{Z}\big)$ is an optimal solution of problem \eqref{subcmot3dre} and hence $\widehat{X}$ is an optimal solution of problem \eqref{subcmot3d}. This completes the proof.

\subsection{Proof of Theorem \ref{thmBCD}}\label{apd-BCD-proof2}

For the ease of applying the convergence results developed in \cite{lt1992convergence}, we first express problem \eqref{subcmot3dredual} in the following compact form:
\begin{equation}\label{dualcomp}
\min\limits_{\bm{\chi}}~\Psi(E\bm{\chi}) + \langle\bm{q},\,\bm{\chi}\rangle \quad \mathrm{s.t.} \quad \bm{\chi}\in\Xi,
\end{equation}
where $\Psi:\mathbb{R}^{n_1n_2n_3}\to\mathbb{R}$ is defined by $\Psi(\bm{y}):=\varepsilon\sum^{n_1n_2n_3}_{i}\exp((y_i-z_i)/\varepsilon)$, $\bm{z}:=\mathrm{vec}(M)\in\mathbb{R}^{n_1n_2n_3}$, $\bm{q} := -\big{[}\bm{b}^{(1)}; \dots; \bm{b}^{(N)}; \mathrm{vec}(U)\big{]}\in\mathbb{R}^{\sum_{i=1}^N m_i+n_1n_2n_3}$, $\bm{\chi} := \big{[}\bm{y}^{(1)}; \dots; \bm{y}^{(N)}; \mathrm{vec}(W)\big{]}\in\mathbb{R}^{\sum_{i=1}^N m_i+n_1n_2n_3}$, $\Xi := \big{\{} \bm{\chi} := \big{[}\bm{y}^{(1)}; \dots; \bm{y}^{(N)}; \mathrm{vec}(W)\big{]}\,:\, W\leq0 \big{\}}$ and
\begin{equation*}
E := \big{[}\mathrm{vec}(A_1^{(1)}),\dots , \mathrm{vec}(A_{m_1}^{(1)}),\dots,  \mathrm{vec}(A_1^{(N)}),\dots , \mathrm{vec}(A_{m_N}^{(N)}),I_{n_1n_2n_3} \big{]}
\in\mathbb{R}^{n_1n_2n_3\times(\sum_{i=1}^N m_i+n_1n_2n_3)}.
\end{equation*}
One can easily verify that $\mathrm{dom}\,\Psi=\mathbb{R}^{n_1n_2n_3}$ is open and $\Psi$ is strictly convex and twice continuously differentiable on $\mathrm{dom}\,\Psi$.

Moreover, the optimal solution set  of problem \eqref{dualcomp} is nonempty (by Proposition \ref{existopt}) and our dual BCD in Algorithm \ref{algBCD} indeed falls into the algorithmic framework in \cite{lt1992convergence} for solving the problem in form of \eqref{dualcomp}. Also, note from \cite[Lemma 3.3]{lt1992convergence} that the set $\big{\{}E\bm{\chi}:\Psi(E\bm{\chi})+\langle\bm{q},\bm{\chi}\rangle \leq \alpha, ~\bm{\chi}\in\Xi\big{\}}$ is compact for any $\alpha\in\mathbb{R}$. Then, one can easily verify that $\nabla^2\Psi(E\bm{\chi}^*)$ is positive definite for any optimal solution $\bm{\chi}^*$ of problem \eqref{dualcomp}. Based on these facts, we can readily apply \cite[Theorem 2.1]{lt1992convergence} to obtain statement (i), i.e., $\bm{\chi}^t  := (\bm{y}^{(1),t},\dots,\bm{y}^{(N),t},W^t) \to \bm{\chi}^*$ R-linearly.

We next prove statement (ii). Let $\big{\{}\big{(}\hat{\bm{y}}^{(1)},\dots,\hat{\bm{y}}^{(N)},\widehat{W}\big{)}\big{\}}$ be the limit of 
$\big{\{}\big{(}\bm{y}^{(1),\ell},\dots,\bm{y}^{(N),\ell},W^{\ell}\big{)}\big{\}}$. Then, one can see from statement (i) that $\big{\{}\big{(}\hat{\bm{y}}^{(1)},\dots,\hat{\bm{y}}^{(N)},\widehat{W} \big{)}\big{\}}$ is an optimal solution of problem \eqref{subcmot3dredual} and further see from Proposition \ref{existopt} that $\widehat{X} := \exp\big{(}\big{(}\, {\textstyle\sum_{i = 1}^N} \cA^{(i,*)}\hat{\bm{y}}^{(i)}+\widehat{W}-M\,\big{)}/\varepsilon\big{)}$
is an optimal solution of problem \eqref{subcmot3d}. Define the mapping $\mathcal{H}: \mathbb{R}^{\sum^N_{i=1}m_i + n_1n_2n_3}\to\mathbb{R}^{n_1n_2n_3}$ by $\mathcal{H}(\bm{\chi}):=\exp((E\bm{\chi}-\bm{m})/\varepsilon)$, whose Jacobian matrix is given by $\mathrm{J}\mathcal{H}(\bm{\chi})=\varepsilon^{-1}\mathrm{Diag}\big{[}\exp((E\bm{\chi}-\bm{m})/\varepsilon)\big{]}E$.
Then, we see that $\bm{x}^\ell:=\mathrm{vec}(X^{\ell})=\mathcal{H}(\bm{\chi}^\ell)$ and $\hat{\bm{x}} :=\mathrm{vec}(\widehat{X})=\mathcal{H}(\widehat{\bm{\chi}})$, where $\bm{\chi}^\ell := \big{[}\bm{y}^{(1),\ell}; \dots; \bm{y}^{(N),\ell}; \mathrm{vec}(W^\ell)\big{]}$ and
$\widehat{\bm{\chi}} := \big{[}\hat{\bm{y}}^{(1)}; \dots; \hat{\bm{y}}^{(N)}; \mathrm{vec}(\widehat{W})\big{]}$. Moreover, we have
\begin{equation}\label{bdbydual}
\begin{aligned}
&\big{\|}\bm{x}^\ell-\widehat{\bm{x}} \big{\|}
= \big{\|}\mathcal{H}(\bm{\chi}^\ell)-\mathcal{H}(\widehat{\bm{\chi}})\big{\|}
= {\textstyle\big\|\big(\int^1_0\mathrm{J}\mathcal{H}\big(\widehat{\bm{\chi}}  + \tau (\bm{\chi}^\ell-\widehat{\bm{\chi}})\big)
\,\mathrm{d}\tau \big)\cdot(\bm{\chi}^\ell-\widehat{\bm{\chi}} )\big\|}  \\
&\leq {\textstyle\big\|\int^1_0\mathrm{J}\mathcal{H}\big{(}\widehat{\bm{\chi}} + \tau (\bm{\chi}^\ell-\widehat{\bm{\chi}})\big{)}
\,\mathrm{d}\tau \big\|\cdot\big{\|}\bm{\chi}^\ell-\widehat{\bm{\chi}}\big{\|}}
\leq {\textstyle\int^1_0\big\|\mathrm{J}\mathcal{H}\big{(}\widehat{\bm{\chi}} + \tau (\bm{\chi}^\ell-\widehat{\bm{\chi}})\big{)}
\big\|\,\mathrm{d}\tau \cdot\big{\|}\bm{\chi}^\ell-\widehat{\bm{\chi}} \big{\|}},
\end{aligned}
\end{equation}
where the second equality follows from the mean-value theorem. Note that
\begin{equation*}
\Psi(E\widehat{\bm{\chi}}) + \langle\bm{q},\widehat{\bm{\chi}}\rangle
\leq \Psi(E\bm{\chi}^\ell) + \langle\bm{q},\bm{\chi}^\ell\rangle
\leq \Psi(E\bm{\chi}^0) + \langle\bm{q},\bm{\chi}^0\rangle, \quad \forall\,\ell\geq0.
\end{equation*}
It then follows from \cite[Lemma 3.3]{lt1992convergence} that $\{E\bm{\chi}^\ell\}$ is bounded. With this fact, one can easily verify that $\left\|\mathrm{J}\mathcal{H}\big{(}\widehat{\bm{\chi}} + \tau(\bm{\chi}^\ell-\widehat{\bm{\chi}})\big{)}\right\|$ is uniformly bounded from above by some constant $L$, i.e., $\big{\|}\mathrm{J}\mathcal{H}\big{(}\widehat{\bm{\chi}} + \tau (\bm{\chi}^\ell-\widehat{\bm{\chi}})\big{)}\big{\|} \leq L$ for all $\ell\geq0$ and $\tau\in [0,1]$. This together with \eqref{bdbydual} and statement (i) prove statement (ii).

\subsection{The dual BCD for the CMOT problem}\label{apdCMOT}

As a special case of problem \eqref{cmotproblem}, the 3-marginal capacity constrained optimal transport problem \eqref{cmotproblem3d} (taking the linear mappings defined in \eqref{eq-specialAj}) has attracted particular attention. In this section, we write down the concrete iterative scheme of the dual BCD in Algorithm \ref{algBCD} for solving \eqref{subcmot3d} with the constraints in \eqref{cmotproblem3d}. We use $\bm{f}$, $\bm{g}$, $\bm{h}$, $W$ to denote Lagrangian multipliers with respect to the following four constraints
\begin{equation*}
\begin{array}{ll}
{\textstyle\sum_{s,t}}X_{rst} = a_r, ~r = 1, \dots, n_1,
&{\textstyle\sum_{r,t}}X_{rst} = b_s, ~s = 1, \dots, n_2, \vspace{2mm} \\
{\textstyle\sum_{r,s}}X_{rst} = c_t, ~t = 1, \dots, n_3,
&X \leq U,
\end{array}
\end{equation*}
respectively. By using similar arguments as in Section \ref{secBCD}, one obtains the dual subproblem:
\begin{equation}\label{subcmot3dredual_1}
\begin{aligned}
\min\limits_{\bm{f},\,\bm{g},\,\bm{h},\,W}\,
R\big{(}\bm{f},\,\bm{g},\,\bm{h},\,W\big{)}
&:= \varepsilon\,{\textstyle\sum_{r,s,t}}\exp\big{(}\big{(}f_r + g_s + h_t + W_{rst} - M_{rst}\big{)}/\varepsilon\big{)} - \langle\bm{f}, \,\bm{a} \rangle \\
&\hspace{0.5cm} - \langle\bm{g}, \,\bm{b} \rangle - \langle\bm{h}, \,\bm{c} \rangle - \langle W, \,U \rangle + \delta_{-}(W),
\end{aligned}
\end{equation}
where $M := C - \varepsilon\log S$. We then apply the BCD method for solving \eqref{subcmot3dredual_1}. Specifically, start from any $(\bm{f}^0,\bm{g}^0,\bm{h}^0,W^0)\in\mathrm{dom}\,R$,
at the $\ell$-th iteration, compute
\begin{equation*}
\begin{aligned}
\bm{f}^{\ell+1} &= \arg\min_{\bm{f}}\,R\big{(}\bm{f}, \,\bm{g}^\ell, \,\bm{h}^\ell, \,W^\ell\big{)},\quad \quad\quad \;
\bm{g}^{\ell+1} = \arg\min_{\bm{g}}\,R\big{(}\bm{f}^{\ell+1}, \,\bm{g}, \,\bm{h}^\ell, \,W^\ell\big{)}, \\
\bm{h}^{\ell+1} &= \arg\min_{\bm{h}}\,R\big{(}\bm{f}^{\ell+1}, \,\bm{g}^{\ell+1}, \,\bm{h}, \,W^\ell\big{)}, \quad
W^{\ell+1} = \arg\min_{W}\,R\big{(}\bm{f}^{\ell+1}, \,\bm{g}^{\ell+1}, \,\bm{h}^{\ell+1}, \,W\big{)}.
\end{aligned}
\end{equation*}
After some manipulations, one can obtain the following explicit iterative scheme:
\begin{equation}\label{algoBCDex1-1}
\begin{aligned}
\bm{f}^{\ell+1} &= \varepsilon\log(\bm{a}) - \varepsilon\log\left(\left[{\textstyle\sum_{s,t}}\exp\big{(}\big{(}g_s^\ell + h_t^\ell + W_{rst}^\ell - M_{rst}\big{)}/\varepsilon\big{)}\right]_{r=1}^{n_1}\right), \\
\bm{g}^{\ell+1} &= \varepsilon\log(\bm{b}) - \varepsilon\log\left(\left[{\textstyle\sum_{r,t}} \exp\left(\big{(}f_r^{\ell+1} + h_t^\ell + W_{rst}^\ell - M_{rst}\big{)}/\varepsilon\right)\right]_{s=1}^{n_2}\right), \\
\bm{h}^{\ell+1} &= \varepsilon\log(\bm{c}) - \varepsilon\log\left(\left[{\textstyle\sum_{r,s}} \exp\big{(}\big{(}f_r^{\ell+1} + g_s^{\ell+1} + W_{rst}^\ell - M_{rst}\big{)}/\varepsilon\big{)}\right]_{t=1}^{n_3}\right), \\
W^{\ell+1} &= \min\left\{\varepsilon\log U + M-\bm{f}^{\ell+1}\otimes\bm{1}_{n_2}\otimes\bm{1}_{n_3}
-\bm{1}_{n_1}\otimes\bm{g}^{\ell+1}\otimes\bm{1}_{n_3} - \bm{1}_{n_1}\otimes\bm{1}_{n_2}\otimes\bm{h}^{\ell+1}, \,0\right\},
\end{aligned}
\end{equation}
where $\bm{1}_{n_i}$ denotes the $n_i$-dimensional vector of all ones for $i=1, \,2, \,3$. Moreover, let $\widetilde{M}:=\exp(-{M/\varepsilon})$, $\tilde{\bm{f}}^\ell:=\exp({\bm{f}^\ell/\varepsilon})$, $\tilde{\bm{g}}^\ell:=\exp({\bm{g}^\ell/\varepsilon})$, $\tilde{\bm{h}}^\ell:=\exp({\bm{h}^\ell/\varepsilon})$ and
$\widetilde{W}^\ell:=\exp({W^\ell/\varepsilon})$. Then, we can equivalently rewrite the iterative scheme \eqref{algoBCDex1-1} as
\begin{equation}\label{algoBCDex2-1}
\begin{aligned}
\tilde{\bm{f}}^{\ell+1} &= \bm{a}\,.\Big{/}\Big(
\Big[\sum_{s,t} (\widetilde{W}^\ell\circ\widetilde{M}\big)_{rst} \,\tilde{g}_s^\ell \,\tilde{h}_t^\ell\Big]_{r=1}^{n_1} \Big), \quad\,\,
\tilde{\bm{g}}^{\ell+1} = \bm{b}\,.\Big{/}\Big(
\Big[\sum_{r,t} (\widetilde{W}^\ell\circ\widetilde{M}\big)_{rst} \,\tilde{f}_r^{\ell+1} \,\tilde{h}_t^\ell\Big]_{s=1}^{n_2}\Big),
\\
\tilde{\bm{h}}^{\ell+1} &= \bm{c}\,.\Big{/}\Big(
\Big[\sum_{r,s} (\widetilde{W}^\ell\circ\widetilde{M}\big)_{rst} \,\tilde{f}_r^{\ell+1}\,\tilde{g}_s^{\ell+1} \Big]_{t=1}^{n_3}\Big), \quad
\widetilde{W}^{\ell+1} = \min\Big{\{}\big{(}U./\widetilde{M}\big{)}\,.\big{/}
\big{(}\tilde{\bm{f}}^{\ell+1}\otimes\tilde{\bm{g}}^{\ell+1}\otimes\tilde{\bm{h}}^{\ell+1}\big{)},
\,1\Big{\}}.
\end{aligned}
\end{equation}
In our numerical experiments conducted in Section \ref{secnum}, we always adopt the iterative scheme \eqref{algoBCDex2-1} since the proximal parameter $\varepsilon$ in our iEPPA does not need to take a small value.

\section{Dykstra's algorithm with KL projections}\label{apdDyKL}

\underline{Dy}kstra's algorithm with \underline{K}ullback-\underline{L}eibler projections (DyKL) \cite{bl2000dykstras} is adapted in \cite{bccnp2015iterative} to solve the following entropic regularized capacity constrained optimal transport problem:
\begin{equation}\label{encapotpro}
\begin{aligned}
&\min\limits_{X\in\mathbb{R}^{m \times n }}\,\langle C, \,X\rangle + \varepsilon{\textstyle\sum^{m}_{s=1}\sum^{n}_{r=1}}X_{rs}(\log X_{rs}-1) \\
&\hspace{0.45cm}\mathrm{s.t.} \hspace{0.5cm} X\bm{1}_{n} = \bm{a}, ~~X^{\top}\bm{1}_{m} = \bm{b},
~~0 \leq X \leq U, \\
\end{aligned}
\end{equation}
where $C\in\mathbb{R}^{m \times n}_+$, $U\in\mathbb{R}^{m \times n}_{++}$, $\bm{a}:=(a_1, \dots, a_{m})^{\top}\in\Sigma_{m}$, $\bm{b}:=(b_1, \dots, b_{n})^{\top}\in\Sigma_{n}$. Recall the definition of the Kullback-Leibler (KL) divergence between $X\in\mathbb{R}^{m\times n}_+$ and $Y\in\mathbb{R}^{m\times n}_{++}$ is given as follows:
\begin{equation*}
\mathbf{KL}(X, \,Y)={\textstyle\sum_{r,s}}\left( x_{rs}\log\left(x_{rs}/y_{rs}\right) - x_{rs} + y_{rs}\right).
\end{equation*}
Moreover, given a convex set $\mathcal{S}\subseteq\mathbb{R}^{m \times n}$ and $Y\in\mathbb{R}^{m\times n}_{++}$, the projection associated with the KL divergence
(called KL projection) is defined as $\mathrm{Proj}^{\mathbf{KL}}_{\mathcal{S}}(Y) := \arg\min\limits_{X\in\mathcal{S}}~\mathbf{KL}(X, \,Y)$.
Thus, problem \eqref{encapotpro} can be equivalently reformulated as
\begin{equation*}
\min_{X\in\mathbb{R}^{m \times n}}~\mathbf{KL}(X,\,K) \qquad \mathrm{s.t.} \qquad X \in \mathcal{S}_1 \cap \mathcal{S}_2 \cap \mathcal{S}_3,
\end{equation*}
where $K:=e^{-C/\varepsilon}$ is the kernel matrix, $\mathcal{S}_1:=\{X\in\mathbb{R}^{m \times n}\,:\,X\bm{1}_{n} = \bm{a}\}$, $\mathcal{S}_2:=\{X\in\mathbb{R}^{m \times n}\,:\,X^{\top}\bm{1}_{m} = \bm{b}\}$ and $\mathcal{S}_3:=\{X\in\mathbb{R}^{m \times n}\,:\,X \leq U\}$.
Then, the DyKL is presented as follows: let $X^{0}=K$, $Q^{0}_1=Q^{0}_2=Q^{0}_3=\bm{1}_m\bm{1}_n^{\top}$, then for $k\geq0$, compute
\begin{equation}\label{Dykentropreform}
\begin{aligned}
\Pi^{k+1}_0 &= X^{k}, \\
\Pi^{k+1}_1 &= \mathrm{Proj}^{\mathbf{KL}}_{\mathcal{S}_{1}}(\Pi^{k+1}_0 \circ Q^{k}_1),
~~Q^{k+1}_1 = Q^{k}_1 \circ \big{(}\Pi^{k+1}_0./\Pi^{k+1}_1\big{)}, \\
\Pi^{k+1}_2 &= \mathrm{Proj}^{\mathbf{KL}}_{\mathcal{S}_{2}}(\Pi^{k+1}_1 \circ Q^{k}_2),
~~Q^{k+1}_2 = Q^{k}_2 \circ \big{(}\Pi^{k+1}_1./\Pi^{k+1}_2\big{)}, \\
\Pi^{k+1}_3 &= \mathrm{Proj}^{\mathbf{KL}}_{\mathcal{S}_{3}}(\Pi^{k+1}_2 \circ Q^{k}_3),
~~Q^{k+1}_3 = Q^{k}_3 \circ \big{(}\Pi^{k+1}_2./\Pi^{k+1}_3\big{)}, \\
X^{k+1} &= \Pi^{k+1}_3,
\end{aligned}
\end{equation}
where $\circ$ denotes the Hadamard product. Note that the above iterative scheme is a slightly different but an equivalent form of the DyKL used in \cite{bccnp2015iterative}. We adapt it here because it is more explicit and convenient for comparison. Moreover, by simple calculations, one can verify that
\begin{equation*}
\begin{aligned}
\Pi^{k+1}_1 &= \mathrm{Proj}^{\mathbf{KL}}_{\mathcal{S}_{1}}(\Pi^{k+1}_0 \circ Q^{k}_1)= \mathrm{Diag}\Big(\bm{a}.\big{/}\big(\big{(}\Pi^{k+1}_0 \circ Q^{k}_1\big{)}\bm{1}_n\big)\Big)\,\big{(}\Pi^{k+1}_0 \circ Q^{k}_1\big{)}, \\
\Pi^{k+1}_2 &= \mathrm{Proj}^{\mathbf{KL}}_{\mathcal{S}_{2}}(\Pi^{k+1}_1 \circ Q^{k}_2)= \big{(}\Pi^{k+1}_1 \circ Q^{k}_2\big{)}\,\mathrm{Diag}\Big(\bm{b}.\big{/}\big(\big{(}\Pi^{k+1}_1 \circ Q^{k}_2\big{)}^{\top}\bm{1}_m\big)\Big), \\
\Pi^{k+1}_3 &= \mathrm{Proj}^{\mathbf{KL}}_{\mathcal{S}_{3}}(\Pi^{k+1}_2 \circ Q^{k}_3)= \min\big{\{}\Pi^{k+1}_2 \circ Q^{k}_3, \,U\big{\}}.
\end{aligned}
\end{equation*}

It is worth noting that the DyKL in \eqref{Dykentropreform} may suffer from severe numerical issues when $\varepsilon$ takes a small value. Thus, one may need to carry out the computations of $\Pi^k_i$ and $Q^k_i$ ($i=1,2,3$) in the log domain to alleviate the numerical instability. Specifically, by taking logarithm on both sides of above equations and letting $\widetilde{X}^k:=\varepsilon\log X^k$, $\widetilde{\Pi}^k_i:=\varepsilon\log \Pi^k_i$, $\widetilde{Q}^k_i:=\varepsilon\log Q^k_i$, $\tilde{\bm{a}}:=\varepsilon\log\bm{a}$, $\widetilde{U}:=\varepsilon\log U$, we obtain after some manipulations that
\begin{equation}\label{Dykentropreformstable}
\begin{array}{llll}
\widetilde{\Pi}^{k+1}_0 = \widetilde{X}^{k}, \\[5pt]
\widetilde{\Pi}^{k+1}_1 = \left[\tilde{\bm{a}}-\varepsilon\log\left(\left[\exp\big{(}\big{(}\widetilde{\Pi}^{k+1}_0 + \widetilde{Q}^{k}_1\big{)}/\varepsilon\big{)}\right]\bm{1}_n\right)\right]\bm{1}_n^{\top} + \widetilde{\Pi}^{k+1}_0 + \widetilde{Q}^{k}_1,  \quad
&\widetilde{Q}^{k+1}_1 = \widetilde{Q}^{k}_1 + \widetilde{\Pi}^{k+1}_0 - \widetilde{\Pi}^{k+1}_1, \\[5pt]
\widetilde{\Pi}^{k+1}_2 = \bm{1}_m\left[\tilde{\bm{b}} - \varepsilon\log\left(\left[\exp\big{(}\big{(}\widetilde{\Pi}^{k+1}_1+\widetilde{Q}^{k}_2\big{)}/\varepsilon\big{)}\right]^{\top}\bm{1}_m\right)\right]^{\top} + \widetilde{\Pi}^{k+1}_1 + \widetilde{Q}^{k}_2, \quad
&\widetilde{Q}^{k+1}_2 = \widetilde{Q}^{k}_2 + \widetilde{\Pi}^{k+1}_1 - \widetilde{\Pi}^{k+1}_2, \\[5pt]
\widetilde{\Pi}^{k+1}_3 = \min\big{\{}\widetilde{\Pi}^{k+1}_2 + \widetilde{Q}^{k}_3, \,\widetilde{U}\big{\}}, \quad
&\widetilde{Q}^{k+1}_3 = \widetilde{Q}^{k}_3 + \widetilde{\Pi}^{k+1}_2 - \widetilde{\Pi}^{k+1}_3, \\[5pt]
\widetilde{X}^{k+1} = \widetilde{\Pi}^{k+1}_3.
\end{array}
\end{equation}
In this stabilization framework, the initialization is set to $\widetilde{X}^{0}=-C$ and $\widetilde{Q}^{0}_1=\widetilde{Q}^{0}_2=\widetilde{Q}^{0}_3=0$. When checking the primal feasibility accuracy, we recover $X^{k+1}$ by setting $X^{k+1} = \exp(\widetilde{X}^{k+1}/\varepsilon)$.

\section{Construction of a tomographic projection}\label{appen3}

Let $p$ be a nonnegative integer. We consider the following four directions
\begin{equation*}
\vec{v} = (1,\,p), ~(1,\,-p), ~(p,1) ~\textrm{and} ~(p,\,-1).
\end{equation*}
Note that when $p\in\{0, 1\}$, we only have two directions. The process to find the projection \(\cA^{(i)}(X) \) along a given direction \(\vec{v}\) is described as follows (see Figure~\ref{fig-proj-5x5} for a concrete example):
\begin{itemize}[leftmargin=0.8cm]
\item[1.] Plot the entries of $X$ as points on the integer grid $\{1,\ldots,n\}\times \{1,\ldots,n\}$.
\item[2.] For each point, draw a line $\mathsf{v}_j$ parallelling to $\vec{v}$, identify all other points for which $\mathsf{v}_j$ passes through.
\item[3.] Take the sum of the entries of $X$ for all points on $\mathsf{v}_j$ to define $(\cA^{(i)}(X))_j$.
\item[4.] Repeat this process until all $\{\mathsf{v}_j\}$ covers the whole grid, i.e., covers all entries of $X$.
\end{itemize}

\definecolor{wwqqcc}{rgb}{0.4,0,0.8}
\definecolor{qqttcc}{rgb}{0,0.2,0.8}
\definecolor{ffzztt}{rgb}{1,0.6,0.2}
\definecolor{aqaqaq}{rgb}{0.6274509803921569,0.6274509803921569,0.6274509803921569}
\definecolor{qqwuqq}{rgb}{0,0.39215686274509803,0}
\definecolor{ffqqqq}{rgb}{1,0,0}
\definecolor{cqcqcq}{rgb}{0.7529411764705882,0.7529411764705882,0.7529411764705882}

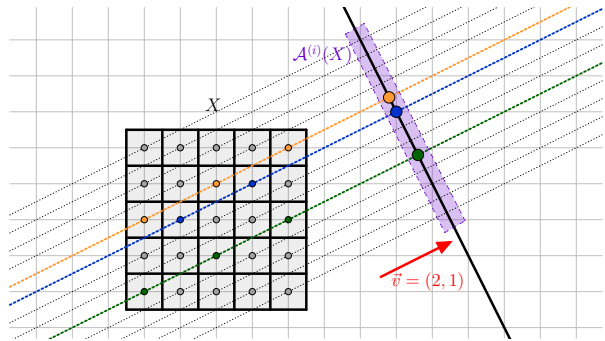
\begin{figure}[H]
	\centering
\resizebox{0.5\textwidth}{!}{
	\begin{tikzpicture}[line cap=round,line join=round,>=triangle 45,x=1cm,y=1cm]
	\draw [color=cqcqcq,, xstep=1cm,ystep=1cm] (-3.737896281689533,-1.2927890742073853) grid (12.719846252396435,7.899823872759559);
	\clip(-3.737896281689533,-1.2927890742073853) rectangle (12.719846252396435,7.899823872759559);
	\fill[line width=2pt,fill=black,fill opacity=0.07] (-0.5,4.5) -- (-0.5,-0.5) -- (4.5,-0.5) -- (4.5,4.5) -- cycle;
	\fill[line width=0.4pt,dash pattern=on 1pt off 1pt on 1pt off 4pt,color=wwqqcc,fill=wwqqcc,fill opacity=0.25] (5.5881578209592,7.1584841189719) -- (8.3459417050493,1.6429163507917) -- (8.9110392295455,1.9703464098485) -- (6.1070956486105,7.455020020487) -- cycle;
	\draw [line width=2pt] (-0.5,4.5)-- (-0.5,-0.5);
	\draw [line width=2pt] (-0.5,-0.5)-- (4.5,-0.5);
	\draw [line width=2pt] (4.5,-0.5)-- (4.5,4.5);
	\draw [line width=2pt] (4.5,4.5)-- (-0.5,4.5);
	\draw [line width=2pt] (0.5,4.5)-- (0.5,-0.5);
	\draw [line width=2pt] (1.5,4.5)-- (1.5,-0.5);
	\draw [line width=2pt] (3.5,4.5)-- (3.5,-0.5);
	\draw [line width=2pt] (2.5,-0.5)-- (2.5,4.5);
	\draw [line width=2pt] (-0.5,3.5)-- (4.5,3.5);
	\draw [line width=2pt] (4.5,2.5)-- (-0.5,2.5);
	\draw [line width=2pt] (-0.5,1.5)-- (4.5,1.5);
	\draw [line width=2pt] (4.5,0.5)-- (-0.5,0.5);
	\draw [->,line width=2pt,color=ffqqqq] (6.570667098709197,0.4114620776103167) -- (8.570667098709198,1.4114620776103166);
	\draw [line width=20pt,color=ffqqqq](6.4,1.0) node[anchor=north west] {\Large $ \vec{v} = (2,1)$};
	\draw [line width=2pt,domain=-3.737896281689533:12.719846252396435] plot(\x,{(-19--2*\x)/-1});
	\draw [line width=0.4pt,dotted,domain=-3.737896281689533:12.719846252396435] plot(\x,{(--8--1*\x)/2});
	\draw [line width=0.4pt,dotted,domain=-3.737896281689533:12.719846252396435] plot(\x,{(--6--1*\x)/2});
	\draw [line width=1.5pt,dotted,color=ffzztt,domain=-3.737896281689533:12.719846252396435] plot(\x,{(--4--1*\x)/2});
	\draw [line width=0.4pt,dotted,domain=-3.737896281689533:12.719846252396435] plot(\x,{(--2--1*\x)/2});
	\draw [line width=1.5pt,dotted,color=qqwuqq,domain=-3.737896281689533:12.719846252396435] plot(\x,{(-0--1*\x)/2});
	\draw [line width=0.4pt,dotted,domain=-3.737896281689533:12.719846252396435] plot(\x,{(-1--1*\x)/2});
	\draw [line width=0.4pt,dotted,domain=-3.737896281689533:12.719846252396435] plot(\x,{(-2--1*\x)/2});
	\draw [line width=0.4pt,dotted,domain=-3.737896281689533:12.719846252396435] plot(\x,{(-3--1*\x)/2});
	\draw [line width=0.4pt,dotted,domain=-3.737896281689533:12.719846252396435] plot(\x,{(-4--1*\x)/2});
	\draw [line width=0.4pt,dotted,domain=-3.737896281689533:12.719846252396435] plot(\x,{(--1--1*\x)/2});
	\draw [line width=1.5pt,dotted,color=qqttcc,domain=-3.737896281689533:12.719846252396435] plot(\x,{(--3--1*\x)/2});
	\draw [line width=0.4pt,dotted,domain=-3.737896281689533:12.719846252396435] plot(\x,{(--5--1*\x)/2});
	\draw [line width=0.4pt,dotted,domain=-3.737896281689533:12.719846252396435] plot(\x,{(--7--1*\x)/2});
	\draw (1.5700963554300853,5.5) node[anchor=north west] {\Large $X$};
	\draw [color=wwqqcc](4,7) node[anchor=north west] {\Large $\mathcal{A}^{(i)}(X)$};
	\draw [line width=0.4pt,dash pattern=on 1pt off 1pt on 1pt off 4pt,color=wwqqcc] (5.5881578209592,7.1584841189719)-- (8.3459417050493,1.6429163507917);
	\draw [line width=0.4pt,dash pattern=on 1pt off 1pt on 1pt off 4pt,color=wwqqcc] (8.3459417050493,1.6429163507917)-- (8.9110392295455,1.9703464098485);
	\draw [line width=0.4pt,dash pattern=on 1pt off 1pt on 1pt off 4pt,color=wwqqcc] (8.9110392295455,1.9703464098485)-- (6.1070956486105,7.455020020487);
	\draw [line width=0.4pt,dash pattern=on 1pt off 1pt on 1pt off 4pt,color=wwqqcc] (6.1070956486105,7.455020020487)-- (5.5881578209592,7.1584841189719);
	\begin{scriptsize}
	\draw [fill=qqwuqq] (0,0) circle (2.5pt);
	\draw [fill=aqaqaq] (0,1) circle (2.5pt);
	\draw [fill=ffzztt] (0,2) circle (2.5pt);
	\draw [fill=aqaqaq] (0,3) circle (2.5pt);
	\draw [fill=aqaqaq] (0,4) circle (2.5pt);
	\draw [fill=aqaqaq] (1,4) circle (2.5pt);
	\draw [fill=aqaqaq] (2,4) circle (2.5pt);
	\draw [fill=aqaqaq] (3,4) circle (2.5pt);
	\draw [fill=ffzztt] (4,4) circle (2.5pt);
	\draw [fill=aqaqaq] (4,3) circle (2.5pt);
	\draw [fill=qqttcc] (3,3) circle (2.5pt);
	\draw [fill=ffzztt] (2,3) circle (2.5pt);
	\draw [fill=aqaqaq] (1,3) circle (2.5pt);
	\draw [fill=qqttcc] (1,2) circle (2.5pt);
	\draw [fill=aqaqaq] (2,2) circle (2.5pt);
	\draw [fill=aqaqaq] (3,2) circle (2.5pt);
	\draw [fill=qqwuqq] (4,2) circle (2.5pt);
	\draw [fill=aqaqaq] (4,1) circle (2.5pt);
	\draw [fill=aqaqaq] (3,1) circle (2.5pt);
	\draw [fill=qqwuqq] (2,1) circle (2.5pt);
	\draw [fill=aqaqaq] (1,1) circle (2.5pt);
	\draw [fill=aqaqaq] (1,0) circle (2.5pt);
	\draw [fill=aqaqaq] (2,0) circle (2.5pt);
	\draw [fill=aqaqaq] (3,0) circle (2.5pt);
	\draw [fill=aqaqaq] (4,0) circle (2.5pt);
	\draw [fill=ffzztt] (6.8,5.4) circle (4.5pt);
	\draw [fill=qqttcc] (7,5) circle (4.5pt);
	\draw [fill=qqwuqq] (7.6,3.8) circle (4.5pt);
	\end{scriptsize}
	\end{tikzpicture}
}
\caption{Construction of the projection operator along \(\vec{v} = (2,1) \)
for a \(5\times 5 \) matrix  \(X\).}
\label{fig-proj-5x5}
\end{figure}

\end{appendices}

\section*{Declarations}

\textbf{Data availability} Not applicable.

\bibliographystyle{plain}
\bibliography{references/Ref_iEPPA}

\end{document}